\documentclass[12pt]{amsart}
\usepackage{amssymb,amsmath,amsthm,bm,mathtools}
 \usepackage[most]{tcolorbox}
\usepackage[colorinlistoftodos,prependcaption,textsize=tiny]{todonotes}
\definecolor{darkblue}{RGB}{0,0,160}
\usepackage[lining]{libertine}   
\usepackage{cabin}
\usepackage[libertine]{newtxmath}

\usepackage[colorlinks=true,allcolors=darkblue]{hyperref}
\usepackage[backrefs,msc-links,nobysame,initials,non-compressed-cites]{amsrefs}

\usepackage[T1]{fontenc}

\usepackage{color}

\def\e{{\rm e}}
\def\cic{\mathbf}
\def\eps{\varepsilon}

\def\d{{\rm d}}

\def\R {\mathbb{R}}
\def\N {\mathbb{N}}

\def\A {{\mathcal A}}

\def\Sy {{\mathsf{Sy}}}

\def\Q {{\mathsf Q}}

\def\M {{\mathrm M}}

\def \l {\langle}
\def \r {\rangle}

\def \and{\qquad\text{and}\qquad}

\newcommand{\supp}{\mathrm{supp}\,}

\newcounter{thms}

\newcounter{other}
\numberwithin{other}{section}
\newtheorem{proposition}[other]{Proposition}
\newtheorem{theorem}[thms]{Theorem}
\newtheorem*{theorem*}{Theorem}
\newtheorem*{proposition*}{Proposition}
\newtheorem{cor}{Corollary}
\newtheorem*{corollary*}{Corollary}
\numberwithin{cor}{thms}
\newtheorem{lemma}[other]{Lemma}
\theoremstyle{definition}

\newtheorem{definition}[other]{Definition}

\def\vint_#1{\mathchoice%
      {\mathop{\kern 0.2em\vrule width 0.6em height 0.69678ex depth -0.58065ex
              \kern -0.8em \intop}\nolimits_{\kern -0.4em#1}}%
      {\mathop{\kern 0.1em\vrule width 0.5em height 0.69678ex depth -0.60387ex
              \kern -0.6em \intop}\nolimits_{#1}}%
      {\mathop{\kern 0.1em\vrule width 0.5em height 0.69678ex depth -0.60387ex
              \kern -0.6em \intop}\nolimits_{#1}}%
      {\mathop{\kern 0.1em\vrule width 0.5em height 0.69678ex depth -0.60387ex
              \kern -0.6em \intop}\nolimits_{#1}}}
\def\vintslides_#1{\mathchoice%
      {\mathop{\kern 0.1em\vrule width 0.5em height 0.697ex depth -0.581ex
              \kern -0.6em \intop}\nolimits_{\kern -0.4em#1}}%
      {\mathop{\kern 0.1em\vrule width 0.3em height 0.697ex depth -0.604ex
              \kern -0.4em \intop}\nolimits_{#1}}%
      {\mathop{\kern 0.1em\vrule width 0.3em height 0.697ex depth -0.604ex
              \kern -0.4em \intop}\nolimits_{#1}}%
      {\mathop{\kern 0.1em\vrule width 0.3em height 0.697ex depth -0.604ex
              \kern -0.4em \intop}\nolimits_{#1}}}

\newcommand{\aveint}[2]{\mathchoice%
      {\mathop{\kern 0.2em\vrule width 0.6em height 0.69678ex depth -0.58065ex
              \kern -0.8em \intop}\nolimits_{\kern -0.45em#1}^{#2}}%
      {\mathop{\kern 0.1em\vrule width 0.5em height 0.69678ex depth -0.60387ex
              \kern -0.6em \intop}\nolimits_{#1}^{#2}}%
      {\mathop{\kern 0.1em\vrule width 0.5em height 0.69678ex depth -0.60387ex
              \kern -0.6em \intop}\nolimits_{#1}^{#2}}%
      {\mathop{\kern 0.1em\vrule width 0.5em height 0.69678ex depth -0.60387ex
              \kern -0.6em \intop}\nolimits_{#1}^{#2}}}

\renewcommand*{\cdots}{%
  \mathinner{{\cdotp}{\cdotp}{\cdotp}}%
}

\numberwithin{equation}{section}

\title[Bilinear Wavelet Representation]{Bilinear Wavelet Representation of Calder\'on-Zygmund Forms}

	\author[F.\ Di Plinio]{Francesco Di Plinio} 
\address[F.\ Di Plinio]{Dipartimento di Matematica e Applicazioni, Universit\`a di Napoli \\ \newline \indent Via Cintia, Monte S.\ Angelo 80126 Napoli, Italy}
\email{{\textnormal{francesco.diplinio@unina.it}}}

\thanks{F. Di Plinio was partially supported by the National Science Foundation under the grant NSF-DMS-2000510.}

\author[A.\ W. Green]{Walton Green} 

\author[B.\ D.\ Wick]{Brett D. Wick}
\thanks{B.\ D.\ Wick's research partially supported in part by NSF grant NSF-DMS-1800057 as well as ARC DP190100970.}

\address[A.\ W. Green, B.\ D.\ Wick]{Department of Mathematics, Washington University in Saint Louis\\ \newline \indent 1 Brookings Drive, Saint Louis, Mo 63130, USA}
\email{\textnormal{francesco.diplinio@wustl.edu, bwick@wustl.edu, awgreen@wustl.edu}}

\subjclass[2010]{Primary: 42B20. Secondary: 42B25}
\keywords{Wavelet representation theorem,  bilinear singular integrals, $T(1)$-theorems, sharp weighted bounds, Leibniz rules, fractional differentiation, sparse domination}

\newcommand{\lip}{\langle}
\newcommand{\rip}{\rangle}
\newcommand{\ip}[2]{\langle #1,#2 \rangle}

\DeclareMathOperator{\MSS}{\Pi}
\DeclareMathOperator{\MSSsig}{{\MSS^\sigma}}
\DeclareMathOperator{\MPS}{{\MSS}}

\DeclareMathOperator{\BMO}{{BMO}}

\newcommand{\cals}{\mathcal{S}}
\renewcommand{\S}{{\mathcal{S}}}
\newcommand{\vl}{\vec \ell}
\newcommand{\varu}{\vartheta}

\begin{document}
\maketitle

\begin{abstract}
We represent a bilinear Calder\'on-Zygmund operator at a given smoothness level as a finite sum of cancellative, complexity zero operators, involving smooth wavelet forms, and continuous paraproduct forms. This representation results in a sparse  $T(1)$-type bound, which in turn yields directly new sharp weighted bilinear estimates  on Lebesgue and Sobolev spaces. Moreover, we apply the representation theorem to study fractional differentiation of bilinear operators, establishing Leibniz-type rules in weighted Sobolev spaces which are new even in the simplest case of the pointwise product.
\end{abstract}


\section{Introduction}
Wavelet decompositions play a central role in the study of singular operators on real variable function spaces.  Haar wavelet techniques finding their roots in the works of Figiel \cite{figiel90}, Nazarov-Treil-Volberg \cite{NTV2}, Petermichl \cite{PetAJM} among others, have led to a powerful and comprehensive theory of   singular integrals on Lebesgue spaces, most prominently for Calder\'on-Zygmund operators (CZOs). Smooth wavelets, sometimes called smooth atoms or molecules, have similarly powered the study of mapping properties of linear and multilinear singular operators on smoothness scales such as the Sobolev, Besov and Triebel-Lizorkin scales, see e.g. \cite{frazier88}, the more recent \cite{HO17} and references therein. Our approach herein seeks to unify these two perspectives.

The driving  result of this article, continuing the theme from \cite{diplinio20}, is a representation of bilinear Calder\'on-Zygmund operators in terms of model operators which reflects the eventual additional smoothness of their off-diagonal kernel. This representation is realized as a sum of continuous paraproduct forms and finitely many cancellative forms which are themselves smooth bilinear  Calder\'on-Zygmund operators.
To wit, the cancellative components of our decomposition, which we term \textit{wavelet forms}, are completely diagonalized forms with respect to a suitable wavelet-type basis. Furthermore, each wavelet form should be 
viewed as a certain approximate projection in the frequency domain. Our prototypical one is
	\begin{equation}\label{eq:intro-wave} U(f,g,h) = \int_0^\infty \int_{\R^d} \l f \otimes g,\nu_{w,t} \r \l h,\phi_{w,t} \r \frac{\d w \, \d t}{t}. \end{equation}
In this formula, $\phi_{w,t} = t^{-d} \phi(\frac{\cdot-w}{t})$ for a smooth wavelet $\phi$, while $\nu_{w,t}$ behaves like the tensor product of two smooth wavelets translated by $w$ and dilated by $t$. We say that $\nu_{w,t}$ belongs to the wavelet class   $\Psi^{k,\delta;1,0}_{w,t}$ defined  below: the  cancellation structure of this class reflects the ``low-high-high'' component of the resolution of the pointwise product. 
The 
diagonal nature of the representation of bilinear CZOs is obtained at the expense of converting the compactly supported wavelets occurring in the resolution of the bilinear identity into non-compactly supported wavelets $\nu_{w,t}$ via the \textit{wavelet averaging lemma}, Lemma \ref{lemma:star} below.

Next, we explain the advantages of our representation. The wavelet and paraproduct model operators are dominated by  intrinsic  localized  forms which in turn satisfy a sharp form of \emph{sparse bounds}. Sparse domination, a technique originating from the early work of Lerner   \cite{lerner-a2-13} and then  developed by several authors within and beyond Calder\'on-Zygmund theory, see e.g. \cite{BC2020,culiuc2018domination,lacey-a2-17,Ler2016}, subsumes the full range, and  the sharp quantification of the weighted norm inequalities for  the operator under sparse control.   Thus, in combination with sparse bounds, our  representation theorem yields a variety of novel bilinear, weighted and sharply quantified $T(1)$-theorems on smoothness spaces. We exemplify this paradigm by the loosely described weighted Sobolev theorem that follows,  summarizing the results of Section \ref{sec:sob}. Let us informally introduce a few definitions.

A bilinear operator $T$ is a $(0,0,\delta)$ CZO if its off-diagonal kernel satisfies   standard bilinear $\delta$-kernel estimates, while $T$ satisfies both the bilinear weak boundedness property and bilinear $T(1)$ testing conditions. These are standard conditions under which the Lebesgue space mapping properties of $T$ are now well understood. We generalize to $(k_1,k_2,\delta)$ CZOs whose kernels are $k_1+k_2$ times differentiable, with appropriate decay estimates and in addition to the weak boundedness testing condition, satisfy an iterative testing condition on monomials $x^\gamma$, producing elements $b_\gamma^i \in \BMO$ for $i=0,1,2$ and $\gamma=(\gamma_1,\gamma_2)$ with $|\gamma_j| \le k_j$.   Sections \ref{subsec:si} and \ref{subsec:cz} contain the the precise definitions.

\begin{theorem*} Let $k_1,k_2 \in \N$,  $\delta>0$,   $0 \le \sigma \le \min\{k_1,k_2\}$,  
$1<p_1,p_2 \le \infty$, $\frac{1}{2}<p\coloneqq\frac{p_1p_2}{p_1+p_2}<\infty$. Let $T$ be a $(k_1,k_2,\delta)$ CZO on $\R^d$ such that $D^{\sigma-|\gamma|}b^0_\gamma \in \BMO$ for all $|\gamma| \le \sigma$.   Suppose that the weight vector $\vec v=(v_1,v_2,v_3)$ satisfies with $\vec p=\left(p_1,p_2,\textstyle \frac{p}{p-1}\right)$
	\[ \begin{array}{c} \sigma \in \{0,1,\ldots,d-1\} \cup [d,\infty) \\ \vec v \in A_{\vec p} \end{array} 
		\quad \mbox{ or } \quad
	\begin{array}{c}  \sigma \in (\frac{d}{p}-d,d) \\ \vec v \in A_{\vec p,\vec r},1 \le r_i < p_i, i=1,2 \\ \frac 1{r_1} + \frac 1{r_2} < \frac{\sigma + d}{d}, r_3=1. \end{array}  \]
Then there exists $C>0$ such that
	\begin{equation}\label{eq:intro} \left\|\frac{D^\sigma T(f_1,f_2)}{v_3}\right\|_{L^p(\R^d)} \le C \|f_1 \|_{W^{\sigma,p_1}(v_1)}\|f_2\|_{W^{\sigma,p_2}(v_2)}.\end{equation}
The constant $C$ depends on the parameters above, the operator $T$, and the appropriate weight characteristic of the weight vector $\vec v$.
\end{theorem*}
{
The bilinear Muckenhoupt  weight vectors, that is,  the classes $A_{\vec p}, A_{\vec p,\vec r} $ appearing in the statement, are  explicitly defined  in Subsection \ref{ss:wclasses}. They were first introduced in \cite{lerner2009new}, and later appeared in \cite{ChTW2017,culiuc2018domination,CDPOBP},  as the natural  multilinear substitute for the role of of the classical $A_p$ in linear Calder\'on-Zygmund theory. In the subsequent articles \cite{nieraeth2019quantitative,li2020extrapolation,Martext2021}, a  complete and useful extrapolation theory for these classes  was developed. These references also contain details on the relation between $ A_{\vec p,\vec r} $  and the linear classes $A_p$. The inhomogeneous weighted Sobolev spaces $W^{\sigma,q}(v)$ are defined by the norm 
	\[ \|f\|_{W^{\sigma,q}(v)} = \left\|[D^\sigma f] v\right\|_{L^q(\R^d)} + \sum_{k=0}^{\lfloor \sigma \rfloor} \|[D^kf] v\|_{L^q(\R^d)}, \quad \]
where $D^\sigma$ is the fractional derivative defined by the Fourier multiplier $m(\xi) = |\xi|^\sigma$.}
The statement above shows how our representation theorem unifies the treatment of smooth CZOs with  fractional differentiation. To see this, note that taking $T$ to be the pointwise product operator returns a form of the Coifman-Meyer-Kenig-Stein-Grafakos-Torres fractional Leibniz rule; cf.\ the excellent survey \cite{Gr2017} and references therein.

Bilinear representations of dyadic-probabilistic type, originating from Hyt\"onen's theorem \cite{hytonen12}, have been developed more recently \cite{li2019bilinear} to obtain  results of this type in the case $\sigma=0$: see also  the recent works \cite{DPLiMV1,DPLiMV2} for multilinear representations in UMD spaces and previous works \cite{li2014sharp,lerner2019intuitive} on sharp weighted norm inequalities for multilinear operators not reliant on representation formulas. When $\sigma$ is a positive integer, boundedness on Sobolev, Besov and Triebel-Lizorkin   spaces is known for certain cancellative CZOs in the Banach range \cite{benyi2003bilinear,maldonado09}. A complete multilinear, sharp weighted theory on fractional smoothness spaces, and with full treatment of the paraproducts was unknown prior to this work. In the previously given example, the dependence of $C$ in (\ref{eq:intro}) upon the characteristic of the weight $\vec v$ is sharply quantified \cite{li2020extrapolation,nieraeth2019quantitative,li2014sharp,lerner2019intuitive}. See the statement of Theorem \ref{thm:sob-min} below for the explicit form.

On the other hand, the fractional derivative $D^\sigma$ applied to $T(f_1,f_2)$ has received renewed interest since its initial study by Kato-Ponce \cite{kato1988commutator} and Kenig-Ponce-Vega \cite{kenig1993well} when Grafakos-Oh \cite{grafakos2014kato} and Muscalu-Schlag \cite{muscalu2013classical} independently extended the $L^{p_1} \times L^{p_2} \to L^p$ results to the sharp bilinear range, taking $p<1$. 
 
Since then,  Leibniz rules  for Fourier multiplier operators and certain pseudodifferential operators have also been obtained. Prior    weighted estimates require the memberships $v_i \in A_{p_i}$ for each single weight \cite{muscalu2013classical,cruz-uribe2016kato,hart2018smoothing,brummer2018bilinear,brummer2019weighted,naibo2019coifman}. This condition is strictly more restrictive than membership of the weight vector $\vec v$ to the multilinear weighted classes    $A_{\vec p} $ required in our theorem, so that strictly speaking (\ref{eq:intro}) is new even when $T$ is the identity. Furthermore, the class of smooth CZOs which we consider includes both smooth Fourier multipliers and certain classes of pseudodifferential operators.

\subsection*{Organization}
The paper is almost entirely self-contained. We only use the well-known principle that sparse domination implies sharp weighted Lebesgue space bounds as a black box. One can consult \cite{li2020extrapolation,nieraeth2019quantitative} for a precise statement, but we also refer to some of the pioneering works \cite{lerner-a2-13,li2014sharp,lerner2019intuitive,conde-alonso2016pointwise} concerning this principle. Otherwise, we do not appeal to dyadic or linear representation or $T(1)$ theorems, abstract sparse domination results, or the Coifman-Meyer multiplier theorem. We will need two technical lemmas from the study of the linear wavelet representation theory in \cite{diplinio20} on the boundedness of the intrinsic square function and the almost-orthogonality of the wavelet classes. 

We begin by recalling the Calder\'on reproducing formula and extending it to a certain multilinear setting, using high-low cancellation. In the same section, we introduce the linear and multilinear wavelet classes, $\Psi_z$, and prove the key wavelet averaging Lemma \ref{lemma:star}. This lemma  allows us to avoid wavelet  operators with \emph{complexity}, in the language of \cite{hytonen12}, completely diagonalizing the CZO.
In Section \ref{sec:rep}, we state the technical definitions of CZOs and higher-order paraproducts, which are smooth testing conditions of $T(1)$-type. After this, we prove the off-diagonal estimates and deduce the representation theorem. Section \ref{sec:sob} is devoted to applications of the representation theorem, specifically to obtain the weighted Sobolev and fractional Sobolev space bounds as a consequence of the sparse domination of the intrinsic forms the latter being proved in Section \ref{sec:sparse}. Section \ref{sec:asym} contains an asymmetric formulation of the results and the extension from bilinear to $m$-linear operators. We conclude with some remarks and further questions in Section \ref{sec:comments}.

\section{Wavelets}\label{sec:wavelets}
To facilitate the description of our wavelet system a few pieces of notation need to be introduced. We work with a fixed dimension $d\geq 1$, thus the space of Schwartz functions $ \mathcal S(\R^d) $ is simply denoted by $\cals$ when no confusion arises.  The Fourier transform  $\mathcal F:\cals \to \cals  $ is normalized as
\[ (\mathcal F \phi)(\xi) =\widehat \phi(\xi) =  \frac{1}{(2\pi)^{\frac d2}} \int_{\R^d} \phi(x)\e^{-ix\xi} \,  \d x .\]
With the above normalization, $ \mathcal F \widehat \phi = \phi (-\cdot)$. 
The affine group $ Z^d\coloneqq \R^d \ltimes (0,\infty) $ acts on $\phi \in L^1_{\mathrm{loc}}(\R^d)$  unitarily by
 \[
 \Sy_z\phi(\cdot) = \phi_{z}(\cdot) \coloneqq\frac{1}{t^d}  \phi\left(\frac{\cdot-w}{t}\right), \qquad z=(w,t)\in Z^d .\] A function $\phi\in  \mathcal S$ is \textit{admissible} if	\begin{equation}\label{eq:adm} \int_0^\infty |\widehat \phi(\rho \xi)|^2 \frac{\d\xi}{\xi}=1 \end{equation}
for all $\rho \in \mathbb{S}^{d-1}$. If $\phi$ is admissible, the
\textit{Calder\'on formula}
\begin{equation}\label{eq:calderon} f = \int_{Z^d} \lip f,\phi_z \rip \phi_z\, \d\mu(z) \qquad \forall f\in \mathcal S  \end{equation}
holds; see e.\ g.\ \cite{frazier91}. Here and in what follows, $\mu$ is the invariant measure on $Z^d$ given by \[
\int\displaylimits_{Z^d}  f(z) \, \d \mu(z) =\int\displaylimits_{\R^d \times(0,\infty)} f(w,t) \, \frac{\d w \d t}{t}, \qquad f\in \mathcal C_{0}(Z^d).
\]
The admissibility condition (\ref{eq:adm}) implies that $\phi$ has mean zero. In general,  our wavelets are required to have more cancellation. Denote by
	\[ \cals_j = \{ \phi \in\cals : \int x^\alpha \phi(x) \, dx = 0 \mbox{ for } 0 \le |\alpha| \le j \}.\]
For functions $\phi \in \cals_j$, and $\gamma \in \N^d$, $0 \le| \gamma| \le j$, define the anti-derivative of order $\gamma$ to be the Fourier multiplier 
	\[ \partial^{-\gamma} \phi(x) = \int_{\R^d} \frac{\xi^\gamma}{|\xi|^{2|\gamma|}} \widehat \phi(\xi) \e^{ix\xi} \, \d\xi, \qquad x\in \R^d.\]
If $\phi \in \cals_j$, then $|\xi|^{-j} \widehat \phi(\xi)$ is bounded for all $\xi$ (specifically close to zero) so the integral defining $\partial^{-\gamma} \phi$ converges absolutely. Denote by $\partial^\gamma$ the usual partial differentiation operator. This  can also be written, for $f \in \cals$,
	\[ \partial^\gamma f(x) = \int_{\R^d} \xi^\gamma \widehat f(\xi)\e^{ix\xi} \, \d \xi, \qquad x\in \R^d.\]
By Plancherel's theorem, for $f \in \cals$ and $\phi \in \cals_j$, the integration by parts formula
	\[ \lip f, \phi \rip = \sum_{|\gamma|=j} \lip \partial^\gamma f, \partial^{-\gamma} \phi \rip \] holds. The symbol
$D^\sigma$ stands for the fractional differentiation operator, namely the Fourier multiplier $m(\xi)=|\xi|^\sigma$ for any $\sigma$ real. We will also utilize the Japanese bracket $\l x \r = 1+|x|$ and the fact that it is equivalent to $\max\{1,|x|\}$ and for $x = (x_1,\ldots,x_m) \in (\R^d)^m$, $\l x \r \sim \max\{1,|x_1|,\ldots,|x_m|\}$. 
\begin{definition}
Let  $D$ be a nonnegative integer. We say $\phi\in \mathcal S_{2D}$ is a mother wavelet if $\phi$ is supported in $B(0,1/2)$, admissible, and 
for all $0 \le |\alpha| \le D$
	\[ \partial^{\pm \alpha} \phi \in \cals_D. \]
If a function is radial, Schwartz, and mean zero, then it only needs to be normalized so that (\ref{eq:adm}) holds. So the admissibility condition can more or less be dropped from the definition. Such wavelets can be constructed as $\Delta^{4D} \Phi$ where $\Phi \in C^\infty_0(B(0,1/2))$.
\end{definition}

Crucial to our program are the functions which behave like wavelets in their scale and decay, similar to the so-called molecules of Frazier, Jawerth, and Weiss \cite{frazier91,frazier88}. Accordingly, introduce the norm
	\begin{equation}\label{eq:star-norm} \|\varphi\|_{\star,\eta,\delta} := \sup_{x \in \R^d} \l x \r^{d+\eta} |\varphi(x)| + \sup_{x \in \R^d, 0 \le |h| \le 1} \l x \r^{d+\eta} \dfrac{|\varphi(x+h)-\varphi(x)|}{|h|^\delta}. \end{equation}
\begin{definition}\label{def:psi}
The wavelet class $\Psi^{k,\delta;1}_z$ is defined by
	\[ \{ \varphi \in \mathcal{C}^k(\R^d) : t^{|\gamma|} \|(\Sy_z)^{-1} \partial^\gamma \varphi \|_{\star,|\gamma|,\delta} \le 1\mbox{ for } 0 \le |\gamma| \le k\} \]
and its cancellative subclass is given by \[
\Psi^{k,\delta;0}_z = \{\varphi \in \Psi^{k,\delta;1}_z : \int x^\alpha \phi(x) =0 \mbox{ for } 0 \le |\alpha| \le k \}.\] Notice that  $\phi_z \in \Psi^{k,1;1}_z$, $\psi_z \in \Psi^{k,1;0}_z$ whenever $\phi\in \cals,\psi\in \cals_k$  are suitably normalized.

The study of bilinear operators requires a suitable tensor-type class $\Psi^{k,\delta;1,1}_{z}$. First, for functions in $L^1_{loc}(\R^d \times \R^d)$,  write $\Sy^j_z$ for the action of $z \in Z^d$ on the $j$-th copy of $\R^d$,  $j=1,2$. Then, $\Psi^{k,\delta;1,1}_z$ is the collection of all $\phi \in  \mathcal{C}^k(\R^{2d})$ which satisfy the estimates
	\[ t^{|\gamma|} \|(\Sy^1_z\Sy^2_z)^{-1} \partial^\gamma \phi \|_{\star,k,\delta} \le 1 \]
for all $\gamma \in \N^{2d}$, $0 \le |\gamma| \le k$. The norm is defined by (\ref{eq:star-norm}) but replacing $\R^d$ with $\R^{2d}$.
\end{definition}

This norm is larger than purely tensoring the norm $\|\cdot\|_{\star,\eta,\delta}$ which, in fact, is not enough for the $L^p$ boundedness of our intrinsic form in the full multilinear range of exponents (see Section \ref{sec:sparse}).
To demonstrate the usefulness of this class, we introduce the cancellative intrinsic forms which will be used in our representation. For $\nu_z \in \Psi^{0,\delta;1,1}_z$ and $\phi \in {\mathcal C_0^\infty}(B(0,1))$, define 
	\[ U(f,g,h) = \int_{Z^d} \l f \otimes g, \nu_{z} \r \l h, \phi_z \r \, \d\mu(z). \]
The form $U$ can be written as $U(f,g,h)=\l K, f \otimes g \otimes h \r$ with kernel
	\[ K(x_0,x_1,x_2) = \int_{Z^d} \nu_z(x_1,x_2) \phi_z(x_0) \, \d\mu(z). \]
For the size estimate on $K$ fix $x_0,x_1,x_2$ and  divide the integration in $t$ into the two regions $L=\{2t< \max_{j}|x_j-x_0|\}, $ $ L^c = (0,\infty) \backslash L.$ Using the fact that
	\[ \nu_z(x_1,x_2) \le \frac{1}{t^{2d} \max\{1,\frac{|x_1-w|}{t},\frac{|x_2-w|}{t}\}^{2d+\delta}} = \frac{t^\delta}{\max_{j \ne 0}\{t,|x_j-w|\}} \]
and that $\phi$ is supported in $B(0,1)$,
	\[ \begin{aligned} |K(x_0,x_1,x_2)| &\le \int_{t \in L} +\int_{t \in L^c} \int_{w \in B(x_0,t)} \frac{t^\delta}{\max_{j \ne 0}\{t,|x_j-w|\}^{2d+\delta}} t^{-d}\phi\left(\frac{x_0-w}{t}\right) \dfrac{\d w \, \d t}{t} \\
		&\lesssim (\max_{j }|x_j-x_0|)^{-2d}. \end{aligned}\]
A similar H\"older estimate can also be proved. More precisely, to use terminology of Section \ref{subsec:si} below, $U$ is a $(0,\delta)$ singular integral (SI) form. If $\nu_z \in \Psi^{k+\delta;1,1}_z$ then $U$ is a $(k,\delta)$ SI form. 

While admissible wavelets themselves satisfy the remarkable orthogonality properties which yield the Calder\'on reproducing formula (\ref{eq:calderon}), the elements of the wavelet class satisfy the following almost orthogonality estimate.
\begin{lemma}\label{lemma:almost}
\cite[Lemma 2.3]{diplinio20} Let $0 < \eta <\delta \le 1$, $0 \le k \le D$, and $s \ge t$. Set $z=(w,t)$ and $\zeta=(v,s)$. Then 
	\[ \sup_{\nu_z \in \Psi_{z}^{k,\delta;0}} \sup_{\theta_\zeta \in \Psi_{\zeta}^{k,\delta;1}} |\l \nu_z,\theta_\zeta \r| \lesssim \frac{t^{k+\eta}}{\max\{s,|v-w|\}^{d+k+\eta}}
	\]
	and for a mother wavelet $\phi$, 
	\[ \sup_{\theta_\zeta \in \Psi_{\zeta}^{k,\delta;1}} |\l \phi_{z}, \theta_\zeta \r | \lesssim \frac{t^{k+\delta}}{\max\{s,|v-w|\}^{d+k+\delta}}. \]
\end{lemma}

In the sequel, we will often denote elements of $\Psi^{k,\delta;i}_z$ or $\Psi^{k,\delta;i,j}_z$ by $\nu_z$ or $\theta_z$. This means only that the function $\nu_z$ is associated to a point $z \in Z^d$, \textit{not} that $\nu_z$ is given by the group action $\Sy_z \nu$ for some function $\nu$. Whether the subscript denotes group action or not will be clear from the context, e.g. if $\phi$ is first introduced and then $\phi_z$, of course $\phi_z$ is the group action. If $\nu_z$ is introduced as an element of $\Psi^{k,\delta;1}_z$ then it is a function associated to $z$. There is even less ambiguity since $\phi_z$ is of course a function associated to $z$.

\subsection{Averaging of Wavelets}
We will use the following \textit{wavelet averaging lemma} to diagonalize the wavelet shifts (the continuous analogue of the Haar shifts). 

\begin{lemma}\label{lemma:star}
Let $0<\eta<\delta$ and $H:\R^{3d} \times (0,\infty)^2 \to \mathbb{C}$ be such that 
	\[ |H(u,v,w,s,t)| \lesssim \dfrac{t^{\delta}}{\max\{s,|u-w|,|v-w|\}^{2d+\delta}}. \]
If $\psi,\phi \in \cals$ then
	\[ \nu_{w,t}(x,y) = \int_t^\infty \int_{\R^{2d}}H(u,v,w,s,t) \psi_{u,s}(x) \phi_{v,s}(y) \dfrac{\d u \, \d v \, \d s}{s} \]
satisfies
	\[ \|(\Sy_{z}^1\Sy_{z}^2)^{-1}\nu_{z}\|_{\star,\eta,1} \lesssim 1. \]
\end{lemma}

\begin{proof} Let us fix $w,t$ and simply write $\nu$ for $\nu_{w,t}$. Make the changes of variable $\alpha=\frac{u-w}{t}$, $\alpha'=\frac{v-w}{t}$ and $\beta=\frac st$. In this way,
\begin{equation}\label{eq:var}  (\Sy_{w,t}^1\Sy_{w,t}^2)^{-1}\nu(x,y) = \int_{1}^\infty \int_{\R^{2d}}  H'(\alpha,\alpha',\beta)\psi'(x)\phi'(y) \dfrac{\d\alpha \, \d \alpha' \, \d\beta}{\beta} \end{equation}
where $\psi'=\Sy_{\alpha,\beta}\Sy_{u,s}^{-1}\psi_{u,s}$, $\phi'=\Sy_{\alpha',\beta}\Sy_{v,s}^{-1}\phi_{v,s}$ and $H' \lesssim \max\{\beta,|\alpha|,|\alpha'|\}^{-(2d+\delta)}$. We have suppressed the dependence of $\psi'$ and $\phi'$ on $\alpha,\alpha',\beta$. 

We first get the size estimate on $\nu$. Since $\phi'$ and $\psi'$ are not assumed to have compact support, we decompose into annuli and divide the scale parameter $\beta$ accordingly.
	\[	L_{j,i} = \{\beta > 1: \beta \le \max\{2^{-(j+2)}|x|,2^{-(i+2)}|y|\} \}, \quad \mbox{and} \quad L_{j,i}^c =(1,\infty) \backslash L_{j,i}. \]
Then, for each $\beta>1$, define the annuli $A_j(x,\beta) = B(x,2^{j+1}\beta) \backslash B(x,2^j\beta)$ for $j \ge 1$ and $A_0(x,\beta) = B(x,2\beta)$.
Let $\alpha \in A_j(x,\beta)$ and $\alpha' \in A_i(y,\beta)$, then 
	\[ \max\{\beta,|\alpha|,|\alpha'|\} \ge \left\{ \begin{array}{cl} \frac 12 \max\{|x|,|y|\} & \mbox{if } \beta \in L_{j,i}, \\ \beta & \mbox{if } \beta \in L_{j,i}^c. \end{array} \right. \]
We obtain for $j \ge 1$, that if $\alpha \in A_j(x,\beta)$, $\beta \in L_{j,i}$ then for any $r>0$
	\[ |\psi'(x)| \le \beta^{-d}(1+\frac{|x-\alpha|}{\beta})^{-d-r} \le \beta^{-d}(1+2^{j})^{-d-r} \le 2^{-jr} (\beta 2^j)^{-d}\]
and similarly for $\phi'$, and for $\alpha' \in A_i(y,\beta)$. The estimate also holds when $i$ or $j$ is $0$ simply because $\phi$ and $\psi$ are bounded.
Thus, for each $j,i \ge 0$,
	\[ \int_{A_j(x,\beta)} \int_{A_i(y,\beta)} \dfrac{|\psi'(x) \phi'(y)| \, \d \alpha' \, \d \alpha}{\max\{\beta,|\alpha|,|\alpha'|\}^{2d+\delta}} \lesssim 
	2^{-r(j+i)}\left\{ \begin{array}{cl} 
		\max\{|x|,|y|\}^{-(2d+\delta)} & \beta \in L_{j,i};\\
		\beta^{-(2d+\delta)} & \beta \in L_{j,i}^c.
		\end{array} \right.\]
Therefore,
	\[\begin{aligned} |(\Sy_{w,t}^1\Sy_{w,t}^2)^{-1}\nu(x,y)| &\lesssim \sum_{i,j=0}^\infty 2^{-r(j+i)} \left[\, 
	\int_{L_{j,i}} \max\{|x|,|y|\}^{-(2d+\delta)} \dfrac{\d \beta}{\beta} 
	+ \int_{L_{j,i}^c} \beta^{-(2d+\delta)} \dfrac{ \d \beta}{\beta} \right]\\
		&\lesssim \sum_{i,j=0}^\infty 2^{-r(j+i)}\left[\frac{\log\max\{2^{-j}|x|,2^{-i}|y|\}}{\max\{|x|,|y|\}^{2d+\delta}} + \max\{2^{-j}|x|,2^{-i}|y|\}^{-(2d+\delta)}\right] \\
		&\lesssim \max\{|x|,|y|\}^{-(2d+\eta)} + \sum_{i,j=0}^\infty 2^{-r(j+i)} 2^{\max(j,i)(2d+\delta)} \max\{|x|,|y|\}^{-(2d+\delta)}.
	\end{aligned} \]
Picking $r>2d+\delta$ guarantees convergence of the sum and concludes the proof of the size estimate.
For the H\"older estimate, let $h = (h_1,h_2) \in \R^{d} \times \R^d$. By (\ref{eq:var}) above,
\begin{align*} 
	|(\Sy_{w,t}^1&\Sy_{w,t}^2)^{-1}\nu(x,y) - (\Sy_{w,t}^1\Sy_{w,t}^2)^{-1}\nu(x+h_1,y+h_2)| \\
	&\le \int_1^\infty \int_{\R^{2d}} \dfrac{1}{\max\{\beta,|\alpha|,|\alpha'|\}^{2d+\delta}}\\
	&\hspace{5ex}\times \left(|\phi'(y)| \cdot|\psi'(x)-\psi'(x+h_1)|+ |\psi'(x+h_1)| \cdot |\phi'(y)-\phi'(y+h_2)|\right) \dfrac{\d\alpha \, \d\alpha' \,\d\beta}{\beta}.
\end{align*}

We will only handle the second term as the first is similar. First, assuming that $|h_1|,|h_2|<1$, we obtain the analogous estimate as above: if $\alpha \in A_j(x,\beta)$ and $\alpha' \in A_i(y,\beta)$, then for any $r>0$,
	\[ |\psi'(x+h_1)| \cdot |\phi'(y)-\phi'(y+h_2)| \lesssim 2^{-r(j+i)} \frac{|h_2|}{(\beta 2^j)^d (\beta 2^i)^d}. \]
Following the remaining steps as above proves the result. If $|h_1|$ or $|h_2|$ are larger than one, the H\"older estimate follows from the size estimate.
\end{proof}

An immediate corollary follows. 
\begin{proposition}\label{prop:psi-nu}
Let $k \in \N$, $0<\eta<\delta \le 1$, and $H:\R^{3d} \times (0,\infty)^2 \to \mathbb{C}$ be such that 
	\[ |H(u,v,w,s,t)| \lesssim \dfrac{t^{k+\delta}}{\max\{s,|u-w|,|v-w|\}^{2d+k+\delta}}. \]
Let $\psi,\phi \in \cals$ and define
	\begin{equation}\label{eq:nu} \nu_{w,t}(x,y) = \int_t^\infty \int_{\R^{2d}}H(u,v,w,s,t) \psi_{u,s}(x) \phi_{v,s}(y) \dfrac{\d u \, \d v \, \d s}{s}. \end{equation}
Then, there exists $C>0$ such that $\nu_{w,t} \in C\Psi^{k,\eta;1,1}_{w,t}$. If moreover $\phi \in \cals_k$, then $\nu_{w,t} \in C\Psi^{k,\eta;1,0}_{w,t}$ and $t^{-|\kappa|}\partial^{-\kappa}_y \nu_{w,t} \in \Psi^{0,\eta;1,0}_{w,t}$ for $|\kappa| \le k$.
\end{proposition}

\begin{proof}
Applying $t^{|\gamma|}\partial^\gamma$ to the formula (\ref{eq:nu}) for $|\gamma| \le k$, one can see that the symbol
	\[ H(u,v,w,s,t)\left(\frac{t}{s}\right)^{|\gamma|} \]
satisfies the condition of Lemma \ref{lemma:star} with $\delta$ replaced by $k+\delta$ which proves the first statement. Similarly, if $\phi \in \cals_k$, then $\partial^{-\kappa} \phi \in \cals$ for $|\kappa| \le k$ and the symbol obtained by applying $t^{-|\kappa|}\partial^{-\kappa}$ to (\ref{eq:nu}) satisfies
	\[ H(u,v,w,s,t) \left(\frac{s}{t}\right)^{|\kappa|} \lesssim \dfrac{t^{k-|\kappa|+\delta}}{\max\{s,|u-w|,|v-w|\}^{2d+k-|\kappa|+\delta}} \lesssim \dfrac{t^{\delta}}{\max\{s,|u-w|,|v-w|\}^{2d+\delta}}. \]
\end{proof}

The next lemma is similar, and will be used to convert a portion of the paraproduct into a wavelet form.

\begin{lemma}\label{lemma:psi-theta}
Let $k \in \N$, $0<\eta<\delta \le 1$, and $G: \R^{3d} \times (0,\infty)^2 \to \mathbb C$ satisfy
	\[ |G(u,v,w,s,t)| \lesssim \dfrac{s^{k+\delta}}{t^{2d+k+\delta}}. \]
Then, for $\|\psi\|_{\star,d+\delta,\delta},\|\phi\|_{\star,d+\delta,\delta} \le 1$, there exists $C>0$ such that
	\[ \theta_{w,t}(x,y) := \int_0^{3t} \int_{|u-w|,|v-w| \le 9t} G(u,v,w,s,t) \psi_{u,s}(x) \phi_{v,s}(y) \dfrac{ \d u \, \d v \, \d s}{s} \in C\Psi^{k,\delta;1,1}_{w,t}.\]
Moreover, if $\psi$ (or $\phi$) has vanishing moments up to $k$, then 
	\[ \theta_{w,t} \in C\Psi^{k,\delta;0,1}_{w,t} (\mbox{or } C\Psi^{k,\delta;1,0}_{w,t}) \quad  \mbox{and}\quad t^{-|\kappa|}\partial^{-\kappa}_x \theta_{w,t} \in C\Psi^{0,\delta;0,1}_{w,t} (\mbox{or }t^{-|\kappa|}\partial^{-\kappa}_y \theta_{w,t} \in C\Psi^{0,\delta;1,0}_{w,t}) \]
for every $|\kappa| \le k$.
\end{lemma}

\begin{proof}
To check the size estimate of $\theta=\theta_{w,t}$, perform the change of variable as in the previous lemma. Assume $|x|>|y|$ and $|x|>18$. In this case,
	\[ |x-\alpha| \ge |x|-|\alpha| \ge |x|-9 \ge \frac 12 |x| \]
so that
	\[ |\psi'(x)| \le \beta^{-d}(1+\frac{|x-\alpha|}{\beta})^{-2d-\delta} \le \beta^{d+\delta}(\beta+|x|)^{-(2d+\delta)}. \]
Thus,
	\[
	\begin{split}
	 |(\Sy_{w,t}^1\Sy_{w,t}^2)^{-1}\partial^{\pm\kappa}\theta(x,y)| \le t^{\mp|\kappa|}\int\displaylimits_{\substack{\alpha\in B(0,9) \\ \alpha' \in B(0,9) \\ 0 < \beta \le 3}} \dfrac{\beta^{k\mp|\kappa|+2\delta}}{(\beta+|x|)^{2d+\delta}} \dfrac{\d \alpha \, \d\alpha' \,\d\beta}{\beta}
	\lesssim \frac{1}{t^{\pm|\kappa|}|x|^{2d+\delta}}.
	\end{split}\]
Symmetry yields the case when $|y|>|x|$. When both $|x|$ and $|y|$ are smaller than 18, we can check the original formula to see
	\[ \theta(x,y) \lesssim \int_0^{3t} s^{k-|\kappa|+\delta}t^{-(k+\delta)} \frac{\d s}{s} \|\psi\|_{L^1} \|\phi\|_{L^1}  \lesssim t^{\mp|\kappa|}. \]
The same method yields the H\"older estimate.
\end{proof}

\subsection{High-Low Cancellation of Wavelets}\label{subsec:calderon}
We return to the Calder\'on formula (\ref{eq:calderon}) from the introduction. In general, it is difficult to analyze operators acting on many different scales at once. It will be helpful in the future to place two functions on the same scale and vary the third. To do so, we use the fact that in the superposition of many wavelets, the smallest scale (highest frequency) dominates. We state this precisely in the following lemma.
\begin{lemma}\label{lemma:calderon}
Let $\phi$ be a radial mother wavelet and $m \ge 0$. There exist functions $\psi^j$, $j=1,2,3,4$, satisfying
	\begin{itemize}
	\item[(i)] $\supp \psi^j \subset B(0,1)$;
	\item[(ii)] $\psi^1,\psi^3 \in C^m$;
	\item[(iii)] $\psi^2,\psi^4 \in \cals_D$;
	\item[(iv)] For any $s>0$ and $f \in L^2(\R^d)$,	\[ \int_{r \ge s} \int_{u \in \R^d} \lip f, \phi_{u,r} \rip \phi_{u,r} \dfrac{ du \ dr}{r} = \int_{\R^d} \lip f, \psi_{u,s}^1 \rip \psi_{u,s}^2 + \lip f,\psi_{u,s}^3 \rip \psi_{u,s}^4 \, du. \]
	\end{itemize}
\end{lemma}
\begin{proof}
Define $\Phi(w) = \int_1^\infty \phi * \phi (\frac{w}{r}) \frac{dr}{r^{d+1}}$. $\Phi$ retains the vanishing moments properties of $\phi$. Since $\phi$ is radial, so is $\phi * \phi$ and thus $\Phi$. Changing the variables and denoting by $p$ the radial function $p(|x|) = \phi * \phi(x)$ (supported in $[0,1)$), we can rewrite
	\[ \Phi(w) = |w|^{-d}\int_0^{|w|} p(\tau) \tau^{d-1} \, d\tau. \] 
Thus, if $w \ge 1$, $\Phi(w) = |w|^{-d} \int_0^1 p(\tau) \tau^{d-1} \, d\tau = |w|^{-d} \int_{\R^d} \phi * \phi = 0$. In particular, $\Phi$ is supported in $B(0,1)$. Next, a few changes of variables yield
	\[  \int_s^\infty \int_{\R^d} \lip f, \phi_{u,r} \rip \phi_{u,r}(x)  \frac{\d u\, \d r}{r}= \lip f, \Phi_{x,s} \rip \]
for any $f \in L^2(\R^d)$ and $x \in \R^d$. We will be done if we can decompose $\Phi = \psi^1 * \psi^2 + \psi^3 * \psi^4$. In that case, changing the variables again,
	\[ \ip{f}{\Phi_{x,s}} = \int_{\R^d}\ip{f}{\psi^1_{u,s}}\psi^2_{u,s}(x) + \ip{f}{\psi^3_{u,s}}\psi^4_{u,s}(x) \, \d u. \]

Decomposing $\Phi$ is difficult if we want the functions to remain Schwartz and compactly supported \cite{yulmukhametov1999solution}.
Obviously $\Phi = \delta_0 * \Phi$. We decompose the delta distribution as follows. Set 
	\[ H(x) = \left\{\begin{array}{cl}\frac{1}{(m+1)!^d} x_1^{m+1}x_2^{m+1}\cdots x_d^{m+1} & x_i \ge 0; \\ 0 & \mbox{else.} \end{array} \right.\]
Setting $D = \frac{\partial}{\partial{x_1}}\frac{\partial}{\partial{x_2}}\cdots\frac{\partial}{\partial{x_d}}$, integrating by parts yields $H * D^{m+2}\Phi = \Phi$. $H \in C^m$ but is not compactly supported. This can be fixed by taking $g \in C^\infty$ such that $g=H$ for $|x| \ge 1/2$. Then, $G:=D^{m+2}g \in C^\infty_0(B(0,1))$. Distributionally, $\delta = D^{m+2}(H-g) + G :=D^{m+2}F + G$.
Therefore,
	\[ \Phi = F * D^{m+2}\Phi + G * \Phi =:\psi^1 * \psi^2 + \psi^3 *\psi^4. \]
\end{proof}

This allows us to obtain the single-scale variant of the bilinear Calder\'on formula.
\begin{lemma}\label{lemma:product}
\begin{align*} f \otimes g = \int_0^\infty \int_{\R^{2d}}  & \lip f, \psi_{u,s}^1 \rip \l g,\phi_{v,s}\r \psi_{u,s}^2 \otimes \phi_{v,s} + \lip f,\psi_{u,s}^3 \rip \l g, \phi_{v,s} \r \psi_{u,s}^4 \otimes \phi_{v,s} \\
		&+ \lip f, \phi_{u,s} \rip \l g,\psi_{v,s}^1 \r \phi_{u,s} \otimes \psi_{v,s}^2+ \lip f,\phi_{u,s} \rip \l g, \psi_{v,s}^3 \r \phi_{u,s} \otimes \psi_{v,s}^4 \, \dfrac{\d u \d v \d s}{s}.
\end{align*}

\end{lemma}
\begin{proof}
Use the Calder\'on formula (\ref{eq:calderon}) on $f$ and $g$ to obtain
	\[ f \otimes g = \int_{Z^d} \int_{Z^d} \lip f,\phi_{u,r} \rip \lip g,\phi_{v,s} \rip \phi_{u,r} \otimes \phi_{v,s} \dfrac{du \, dr \, dv \, ds}{rs}\] 
Split the integral into $r \ge s$ and $s > r$. On the first one, apply Lemma \ref{lemma:calderon} with $f$ and on the second, apply it to $g$. 
\end{proof}

\section{Representation Formula}\label{sec:rep}

\subsection{Singular Integrals}\label{subsec:si} Let $\vec 1_{d}=(1,\ldots, 1)\in \R^d$. 
Given $k \in \N$, a function $K\in L^1_{\mathrm{loc}}( \R^{3d}\setminus \R\cic{1}_{3d})$ is a $(\vl,\delta)$ SI (singular integral) kernel if there exist $C,\delta>0$ such that for all $0 \le |\kappa| \le k$, 
	\[ |\partial_{x_i}^\kappa K(x_0,x_1,x_2)| \le \dfrac{C}{(\sum_{j \ne i}|x_i-x_j|)^{2d+|\kappa|}} \]
	\[ |\partial_{x_i}^\kappa \Delta^i_h K(x_0,x_1,x_2)| \le \dfrac{C|h|^\delta}{(\sum_{j \ne i}|x_i-x_j|)^{2d+|\kappa|+\delta}}.\]
$\Delta_h^i$ denotes the difference operator in the $i$th position.
We say $\Lambda$ is a $(k,\delta)$ trilinear SI form if
	\[ \int K(x_0,x_1,x_2) f(x_1) g(x_2) h(x_0)\, \d x = \Lambda(f,g,h) \]
for all $f,g,h \in\cals$ with $\supp f \cap \supp g \cap \supp h = \varnothing$ and for a $(k,\delta)$ SI kernel $K$. Notice that a $(k,\delta)$ SI form is a $(k',\delta')$ form for any $k' \le k$ and $\delta' \le \delta$.

\subsection{Calder\'on-Zygmund Forms}\label{subsec:cz}
Our representation formula will be built using the following intrinsic singular integral forms.

\subsubsection{Wavelet Forms}
\begin{definition}
A trilinear form $U$ is called a $(k,\delta)$-smooth wavelet form if for each $z \in Z^d$, there exists $\nu_z \in \Psi_{z}^{k,\delta;1,0}$ such that
	\[ U(\pi(f,g,h)) = \int_{Z^d} \lip f \otimes g , \nu_z \rip \lip h,\phi_z \rip \, \d \mu(z) \]
for some permutation $\pi \in S^3$ and a mother wavelet $\phi$.
\end{definition}

\subsubsection{Paraproducts}\label{subsec:para}
Let $\{\theta^\gamma_{z} \in \Psi^{D,\delta;1}_{z}\}_{z \in Z^d}$ be a $\gamma$-family, which means
	\[ \int x^\beta \theta_z^\gamma(x) \, dx = t^{|\beta|}\delta_{\beta,\gamma} \]
for each $|\beta|\le|\gamma| \le D$. These can be constructed by taking a single function $\varu^\gamma$, smooth and compactly supported such that
	\[ \int_{\R^d} \varu^\gamma(x) x^\beta \, \d x = \delta_{\beta,\gamma}, \quad 0 \le \beta \le \gamma,\] and then acting on $\varu^\gamma$ with the affine group $Z^d$, yielding $\varu_z^\gamma$. Such functions $\varu^\gamma$ do indeed exist, see \cite{alpert1993class,rahm2021weighted}. 

Given a function $b \in \BMO$ and multi-index $\gamma \in \N^{2d}$, define the $\gamma$-order paraproduct form 
	\begin{equation}\label{eq:para-form} \Pi_{b,\gamma}(f,g,h)=\int_{Z^d}\lip b,(\partial^{-\gamma_1-\gamma_2}\phi)_{z} \rip \lip f , \varu_z^{\gamma_1} \rip \l g, \varu_{z}^{\gamma_2} \rip \lip h, \phi_{z} \rip \, \d\mu(z) \end{equation}
where $\phi$ is a mother wavelet and $\varu_{z}^{\gamma_\ell}$ are compactly supported $\gamma_\ell$ families. $\Pi_{b,\gamma}$ is a $(M,\delta)$ SI form for any $M>0$ up to the smoothness of $\varu^\gamma$ and $\phi$ and any $0 < \delta \le 1$. This can be verified using the same reasoning in the discussion after Definition \ref{def:psi} of the wavelet classes, only it is simpler since the wavelets are compactly supported. 

We will use the partial ordering on multi-indices $\vec j, \vec k \in \N^{m}$: $\vec j = (j_1,j_2,\ldots,j_m) \le \vec k = (k_1,k_2,\ldots,k_m)$ if $j_\ell \le k_\ell$ and $\vec j < \vec k$ if $\vec j \le \vec k$ but $\vec j \ne \vec k$.
In this way, for all $\kappa \le \gamma$ and $\phi$ with vanishing moments up to $|\gamma|$,
	\begin{align*} \Pi_{b,\gamma}(x^{\kappa_1},y^{\kappa_2},\phi) &= \delta_{\kappa,\gamma}\lip b,\partial^{-\gamma_1-\gamma_2}\phi \rip, \\
	\Pi^{i*}_{b,\gamma}(x^{\kappa_1},y^{\kappa_2},\phi) &=0
	\end{align*}
where $\Pi_b^{*1}(f,g,h) = \Pi_b(h,g,f)$ and $\Pi_b^{*2}(f,g,h) = \Pi_b(f,h,g)$.
In general, this is difficult to compute for $\kappa > \gamma$ (see Section \ref{sec:comments} below). Now we will iteratively define the $\gamma \in \N^{2d}$ order paraproducts of a form $\Lambda$.

Recall $\cals_j = \{ \psi \in \cals : \int x^\gamma \psi(x) \, dx =0 \mbox{ for } |\gamma| \le j \}$.
When $\gamma=0$, we say a $(0,\delta)$ SI form $\Lambda$ has $0$th order paraproducts if there exists BMO functions $b_0^i$, $i=0,1,2$ such that for all $\psi \in \cals_0$,
	\begin{align*} 
			\Lambda(1,1,\psi) &= \lip b_0^0,\psi \rip, \\
			\Lambda(\psi,1,1) &= \lip b_0^1,\psi \rip, \\
			\Lambda(1,\psi,1) &= \lip b_0^2,\psi \rip.
	\end{align*}
This is the standard bilinear $T(1,1)$ condition \cite{christ1987polynomial,li2019bilinear}. 
Now, for $(k_1,k_2)>0$, we define the $(k_1,k_2)$-th order paraproducts inductively. Suppose $\Lambda$ has paraproducts $b_\gamma^{i}$ for all $(|\gamma_1|,|\gamma_2|) < (k_1,k_2)$. Then, we say $\Lambda$ has $(k_1,k_2)$-th order paraproducts if for each $|\gamma_1|=k_1$, $|\gamma_2|=k_2$, there exist $b_{\gamma}^i \in \BMO$ such that for all $\psi \in \cals_{k_1+k_2}$,
	\[ \Lambda_{k_1,k_2} := \Lambda - \sum_{i=0}^2 \sum_{(|\kappa_1|,|\kappa_2|) < (k_1,k_2)} \Pi^{i*}_{{b_\kappa^i},\kappa} \]
satisfies
	\begin{align*} \Lambda_{k_1,k_2}(x^{\gamma_1},y^{\gamma_2},\psi) &= \lip b_\gamma^0,\partial^{-\gamma_1-\gamma_2}\psi \rip; \\
	\Lambda_{k_1,k_2}(\psi,y^{\gamma_2},x^{\gamma_1}) &= \lip b_\gamma^1,\partial^{-\gamma_1-\gamma_2}\psi \rip; \\
	\Lambda_{k_1,k_2}(x^{\gamma_1},\psi,y^{\gamma_2}) &= \lip b_\gamma^2,\partial^{-\gamma_1-\gamma_2}\psi \rip. 
	\end{align*}
Under this definition, one can verify by induction that $\Lambda_{k_1,k_2}$ has vanishing paraproducts of all orders $<(k_1,k_2)$. The action of $(k,\delta)$ SI forms on polynomials of degree $(k_1,k_2)$ with $k_1+k_2 \le k$ can be defined as elements of the dual space of $\cals_{k_1+k_2}$, see \cite{frazier88,benyi2003bilinear}.

\begin{definition}\label{def:cz}
Let $k_1+k_2 \le k$. A $(k,\delta)$ SI form $\Lambda$ is called a $(k_1,k_2,\delta)$ Calder\'on-Zygmund (CZ) form if it has paraproducts up to order $(k_1,k_2)$ and satisfies the Weak Boundedness Property (WBP), which means
	\begin{equation}\label{eq:wbp-1} t^{2d}|\Lambda(\phi_z,\psi_z,\varu_z)| \le C \end{equation}
for all $\phi_z,\psi_z,\varu_z \in \Psi_{z}^{0,\delta;1,1}$ supported in $B(w,t)$ and $(w,t)=z \in Z^d$. We also say $T$ is a $(k_1,k_2,\delta)$ Calder\'on-Zygmund operator (CZO) if $\Lambda(f,g,h) = \l T(f,g),h \r$ is a $(k_1,k_2,\delta)$ CZ form.
\end{definition}

\subsection{Smooth Representation Theorem}
\begin{theorem}\label{thm:sym}
Let $\Lambda$ be a $(k_1,k_2,\delta)$ CZ form and $\eta < \delta$. There exists $(j,\eta)$-smooth wavelet forms $U^i_j$, for $j=\min\{k_1,k_2\},\ldots,k_1+k_2$, $i=1,\ldots,6$ and paraproduct forms $\Pi_{b_\gamma^i,\gamma}^{i*}$
$i=0,1,2$ such that
	\[ \Lambda(f,g,h) = \sum_{i=1}^6 \sum_{j=\min\{k_1,k_2\}}^{k_1+k_2} U_j^i(f,g,h) + \sum_{i=0}^2 \sum_{|\gamma_\ell|\le |k_\ell|}\Pi^{i*}_{b_\gamma^{i*},\gamma}(f,g,h). \]
\end{theorem}

The region of interest $Z(w,t) := \R^{2d} \times (t,\infty)$ will be partitioned into the following regions (Far, Near, High-Low):
\begin{equation}\label{eq:regions}\begin{aligned} F(w,t) &:= \{\max\{|u-w|,|v-w|\} \ge 3s, \, s \ge t\}, \\
				S(w,t) &:= \{\max\{|u-w|,|v-w|\} \le 3s, \, s \le 3t\}, \\
				\mbox{and} \quad A(w,t) &:= \{\max\{|u-w|,|v-w|\} \le 3s, \, s \ge 3t\}.
\end{aligned} \end{equation}
We will also use the region $I(w,t) := \{ \max\{|u-w|,|v-w|\} \le 3\max\{s,t\}, \, s \le 3t\}$.

\subsubsection{Kernel Estimates}
\begin{lemma}\label{lemma:U}
 Let $\Lambda$ be a $(k_1,k_2,\delta)$ CZ form, $\phi$ be a mother wavelet, and $\psi \in {\mathcal C_0^\infty}(B(0,1))$. For each $0\le|\gamma|\le \max\{k_1,k_2\}$, let $\varu_{w,t}^{\gamma}$ be a $\gamma$-family. Define
	\[ P_{k}(x) = \sum_{0 \le |\gamma| \le k} \lip \psi_{u,s}, \varu^{\gamma}_{w,t} \rip \left(\frac{x-w}{t}\right)^\gamma\]
and $\tilde P_k$ replacing $\psi_{u,s}$ with $\phi_{v,s}$. For $s \ge t$, define
	\begin{equation}\label{def:upsilon} \begin{aligned} \Upsilon(u,v,w,s,t) = \Lambda(\psi_{u,s},&\phi_{v,s},\phi_{w,t}) - 1_{A(w,t)}(u,v,s)\left[ \Lambda(P_{k_1},\tilde P_{k_2},\phi_{w,t})\right. \\
		&\left.+ \Lambda(P_{k_1-1},\phi_{v,s}-\tilde P_{k_2},\phi_{w,t}) -\Lambda(\psi_{u,s}-P_{k_1},\tilde P_{k_2-1},\phi_{w,t}) \right]. \end{aligned}\end{equation}
Then, for any $\eta<\delta$, \[ |\Upsilon(u,v,w,s,t)| \lesssim_\eta \dfrac{t^{k_1+k_2+\eta}}{\max\{s,|u-w|,|v-w|\}^{2d+k_1+k_2+\eta}}. \]

\end{lemma}
\begin{proof}
\textit{Region 1:} When $(u,v,s) \in F(w,t)$, $\max\{|v-w|,|u-w|\} \ge 3s$.
Integrating by parts $k_1+k_2$ times, rewrite $\Lambda(\psi,\phi,\phi) = \sum_{|\gamma|=k_1+k_2}\lip \partial_{x_0}^\gamma K, \psi\otimes \phi \otimes \partial^{-\gamma}(\phi) \rip$. Use the H\"older estimate of the kernel and the fact that $\partial^{-\gamma} (\phi_{w,t}) = t^{|\gamma|} (\partial^{-\gamma} \phi)_{w,t}$ to get
	\[ t^{k_1+k_2}\left| \int_{x_1 \in B(u,s)} \int_{x_2 \in B(v,s)} \dfrac{t^\delta \psi_{u,s}(x_1)\phi_{v,s}(x_2) \d x_1 \d x_2}{(|x_1-w|+|v-w|)^{2d+k_1+k_2+\delta}} \right|\le \dfrac{t^{k_1+k_2+\delta}}{\max\{|u-w|,|v-w|\}^{2d+k_1+k_2+\delta}}. \]
It was also important here that $\partial^{-\gamma}\phi$ was still mean zero.

\textit{Region 2:} In the region $S(w,t)$, appealing to the WBP gives the estimate of $t^{-2d}$.

\textit{Region 3:} In the final region, $A(w,t)$, $|u-w|,|v-w| \le 3s$ and $s \ge 3t$.
Here we rewrite
	\begin{equation}\Lambda( \psi_{u,s},\phi_{v,s},\phi_{w,t}) - \Lambda(P_{k_1},\tilde P_{k_2},\phi_{w,t}) \label{eq:ups}\end{equation}
	\[ = \Lambda( \psi_{u,s}-P_{k_1},\phi_{v,s}-\tilde P_{k_2},\phi_{w,t}) + \Lambda(P_{k_1},\phi_{v,s}-\tilde P_{k_2},\phi_{w,t}) + \Lambda(\psi_{u,s}-P_{k_1},\tilde P_{k_2},\phi_{w,t}). \] 
$P$ and $\tilde P$ satisfy the estimate
	\[|\psi_{u,s}(x)-P_k(x)|,|\phi_{v,s}(x)-\tilde P_k(x)| \le \dfrac{1}{s^d} \left(\frac{|x-w|}{s}\right)^k \min\left\{1,\frac{\max\{|x-w|,t\}}{s}\right\}.\]
See \cite[Lemma 3.1]{diplinio20} where this is derived using Taylor polynomials in conjunction with Lemma \ref{lemma:almost}.

Let $\alpha$ be a cutoff function around $B(w,t)$ and set $\Xi_u=\psi_{u,s}-P$ and $\Xi_v=\phi_{v,s}-\tilde P$. We decompose the first term in (\ref{eq:ups}), $\Lambda(\Xi_u,\Xi_v,\phi_{w,t})$, by
\[ T(\Xi_u, \Xi_v) = T(\alpha \Xi_u, \alpha \Xi_v) + T(\alpha \Xi_u, (1-\alpha)\Xi_v) + T(1-\alpha)\Xi_u, \Xi_v) .\]
On the first term, use WBP around $(w,t)$ to get $t^d t^{-d} \|\psi-P\|_{L^\infty (B(w,t))} \|\phi-\tilde P\|_{L^\infty(B(w,t)} \sim \frac{t^{k_1+k_2+2}}{s^{2d+k_1+k_2+2}} \le \frac{t^{k_1+k_2+\delta}}{s^{2d+k_1+k_2+\delta}}$ for any $\delta \le 2$. 
For the third (and second) terms, use the H\"older kernel estimate as above after integrating by parts $k_1+k_2$ times to get 
    \begin{equation}\label{eq:XiXi} \int\limits_{\substack{x_1 \in B(w,t)^c\\ x_2 \in \R^d}}\dfrac{t^{k_1+k_2+\delta} |\Xi_u(x_1)\Xi_{v}(x_2)|}{(|x_1-w|+|x_2-w|)^{2d+k+\delta}} \d x_1 \, \d x_2 \lesssim \dfrac{t^{k_1+k_2+\delta}}{s^{d}} \int\limits_{x_1 \in B(w,t)^c}\dfrac{|\Xi(x_1)| \, \d x_1}{|x_1-w|^{d+k_1+k_2+\delta}},\end{equation}
	where we have used 
		\[ \begin{aligned} \int_{\R^d} \dfrac{|\Xi_{v}(x_2)|}{(|x_1-w|+|x_2-w|)^{2d+k_1+k_2+\delta}} \d x_2 &\le s^{-d-k_2}\int_0^\infty (\tau+|x_1-w|)^{-(2d+k_1+k_2+\delta)}\tau^{d+k_2-1} \, \d\tau \\
		&\le s^{-d-k_2} c_d |x_1-w|^{-(d+k_1+\delta)}. \end{aligned}\]
    Break up the remaining integral in (\ref{eq:XiXi}) into $t \le |x_1-w| \le s$ and $s < |x_1-w|$. In the first case, we have the estimate $|\Xi_u(x_1)|\le 1^{1-\delta}|x_1-w|^{k_1+\delta} s^{-d-k_1-\delta}$. Thus,
    \[ \int_{t \le |x_1-w| \le s} \dfrac{|\Xi_u(x_1)| \, \d x_1}{|x_1-w|^{d+k_1+\delta}} \le \dfrac{1}{s^{d+k_1+\delta}}\int_t^s \tau^{-1} \, d\tau \le \dfrac{1}{s^{d+k+\delta}} \log\left(\frac st\right). \]
    On the other hand, $\int_{s < |x_1-w|} \Xi_u(x_1)|x_1-w|^{-(d+k_1+\delta)} \, \d x_1 \le s^{-d-k_1} \int_s^\infty \tau^{-(1+\delta)} \, d\tau \le s^{-(d+k_1+\delta)}$. 

The remaining terms in (\ref{eq:ups}) are $\Lambda(\psi_{u,s}-P_{k_1},\tilde P_{k_2},\phi)$ and $\Lambda(P_{k_1},\phi_{v,s}-\tilde P_{k_2},\phi)$. Comparing (\ref{eq:ups}) with (\ref{def:upsilon}), we see that we only need to estimate two terms of the form
	\begin{equation}\label{eq:half} \Lambda(\Xi,\tilde P,\phi_{w,t}) = \sum_{|\gamma|=k_2}\lip \psi_{u,s}, \varu_{w,t}^\gamma \rip \Lambda(\Xi,p^{\gamma}_{w,t},\phi_{w,t}) \end{equation}
where $p^{\gamma}_{w,t}(x) = (\frac{x-w}{t})^\gamma$. We will need estimates for each summand in the future for all $\gamma$, so we will estimate the general form $\Lambda(\Xi,p^\gamma,\phi)$. As before, decompose
\[ \Lambda(\Xi, p,\phi_{w,t}) = \Lambda(\alpha \Xi, \alpha p,\phi_{w,t}) + \Lambda(\alpha \Xi,(1-\alpha)p,\phi_{w,t}) + \Lambda((1-\alpha)\Xi, p,\phi_{w,t})\]
where $\alpha$ is a smooth cutoff around $B(w,t)$. 

For the first term, use WBP to get $t^d t^{-d} \|p\|_{L^\infty(B(w,t))}\|\psi-P\|_{L^\infty (B(w,t))}\lesssim \frac{t^{k_1+1}}{s^{d+k_1+1}}$. 
For the second term, follow the outline above integrating by parts $k_1+|\gamma|$ times and applying the H\"older kernel estimate to get
	\begin{align*} |\Lambda(\Xi,p^\gamma,\phi)| &\lesssim t^{k_1+|\gamma|+1}\int_{x_1 \in B(w,t)} \int_{x_2 \in B(w,t)^c} \dfrac{|\Xi(x_1) p^\gamma(x_2)|}{(|x_1-w|+|x_2-w|)^{2d+k_1+|\gamma|+\delta}} \, \d x_1 \, \d x_2\\
			& \le \dfrac{t^{k_1+|\gamma|+k_1+1+\delta}}{s^{d+k_1+\delta}} \int_{x_2 \in B(w,t)^c} \dfrac{|p^\gamma(x_2)| \, \d x_2}{|x_2-w|^{d+k_1+|\gamma|+\delta}}   \le \dfrac{t^{k_1+|\gamma|+k_1+2\delta}}{s^{d+k_1+\delta}} \dfrac{1}{t^{k_1+\delta}} = \dfrac{t^{k_1+1}}{s^{d+k_1+1}}. 
	\end{align*}
For the second inequality, we used the estimate $|\Xi(y)| \le \frac{t^{k_1+1}}{s^{d+k_1}}$ when $|y-w| \le t$.
The third term in the decomposition is similar, but more closely follows the line of proof used on $\Lambda(\Xi,\Xi,\phi)$ above. Thus, the summands in (\ref{eq:half}) have the estimates
		\begin{equation}\label{eq:half-est} \left| \lip \psi_{u,s}, \varu^{\gamma}_{w,t} \rip \Lambda(p^\gamma_{w,t},\phi_{v,s}-P_{k_2},\phi_{w,t}) \right| \le \dfrac{t^{|\gamma|+k_1+\eta}}{s^{2d+|\gamma|+k_1+\eta}} \end{equation}
since the coefficients $|\lip \psi_{u,s} , \varu_{w,t}^\gamma \rip| \lesssim \frac{t^{|\gamma|}}{s^{d+|\gamma|}}$ (see Lemma \ref{lemma:almost} using the fact that $\varu_{w,t}^\gamma$ has vanishing moments up to $|\gamma|$ and H\"older exponent $\delta=1$). Taking the case $|\gamma|=k_2$, we see that the remaining terms in $\Upsilon$, (\ref{eq:half}), satisfy the estimates claimed in the Lemma.
\end{proof}

In the proof of the representation theorem, we will still have to deal with the error terms subtracted off of $\Upsilon$ in the region $A(w,t)$. $\Lambda(P,P,\phi)$ is controlled by the paraproducts assumption, but the so-called half-paraproducts $\Lambda(P_{k_1-1},\phi-P,\phi)$, satisfy the worse estimates (\ref{eq:half-est}) with $|\gamma| \le k_2-1$.

\subsubsection{Proof of Theorem \ref{thm:sym}}
First, we represent $\Lambda$ at the $(k_1,k_2)$th level under the assumption it has vanishing paraproducts of all orders less than $(k_1,k_2)$. We decompose $\Lambda(f,g,h)$ by applying the results of Subsection \ref{subsec:calderon} to $f$, $g$, and $h$. By Calder\'on's formula, we obtain
\begin{align*}
	\Lambda(f,g,h) 
		&=\int_{Z^d \times Z^d \times Z^d} \lip f,\phi_{u,r} \rip \lip g,\phi_{v,s} \rip \lip h,\phi_{w,t} \rip \Lambda(\phi_{u,r},\phi_{v,s},\phi_{w,t}) \d \mu(w,t) \, \d \mu(v,s) \, \d \mu(u,r) \\
		&=\int_{u,v,w} \left(\int_{r,s\ge t>0} + \int_{s,t \ge r >0} + \int_{t, r \ge s >0} \right).
\end{align*}
Split the first integral as $\int_{r \ge s \ge t} + \int_{s \ge r \ge t}$ and use Lemma \ref{lemma:calderon} above to get
	\begin{align*} \int\limits_{\substack{(u,v,w) \in \R^{3d} \\ r,s \ge t > 0}} &= \int\limits_{\substack{(v,w) \in \R^{2d} \\ s \ge t>0}} \Lambda\Bigg(\int\limits_{\substack{u \in \R^d \\ r \ge s}}\lip f,\phi_{u,r} \rip \phi_{u,r}  \d \mu(u,r), \phi_{v,s},\phi_{w,t}\Bigg) \lip g,\phi_{v,s} \rip \lip h,\phi_{w,t} \rip \d \mu(v,s) \, \d \mu(w,t)\\
	&\quad+ \int\limits_{\substack{(u,w) \in \R^{2d} \\ r \ge t}} \Lambda\Bigg(\phi_{u,r},\int\limits_{\substack{ v \in \R^d \\ s \ge r}}\lip g,\phi_{v,s} \rip \phi_{v,s} \, \d\mu(v,s), \phi_{w,t} \Bigg) \lip f,\phi_{u,r} \rip \lip h,\phi_{w,t} \rip \, \d\mu(u,r) \, \d\mu(w,t)\\
	&= \int\limits_{\substack{(u,v,w) \in \R^{3d} \\ s \ge t >0}} \lip f, \psi_{u,s}^1 \rip \lip g, \phi_{v,s}\rip \lip h,\phi_{w,t} \rip \Lambda(\psi_{u,s}^2,\phi_{v,s},\phi_{w,t}) \dfrac{ \d s \, \d t \, \d u \, \d v \, \d w}{st} \\
	&\quad +\int\limits_{\substack{(u,v,w) \in \R^{3d} \\ s \ge t >0}} \lip f, \psi_{u,s}^3 \rip \lip g, \phi_{v,s}\rip \lip h,\phi_{w,t} \rip \Lambda(\psi_{u,s}^4,\phi_{v,s},\phi_{w,t}) \dfrac{ \d s \, \d t \, \d u \, \d v \, \d w}{st}\\
	&\quad+ \int\limits_{\substack{(u,v,w) \in \R^{3d} \\ r \ge t >0}} \lip f, \phi_{u,r} \rip \lip g, \psi_{v,r}^1\rip \lip h,\phi_{w,t} \rip \Lambda(\phi_{u,r},\psi_{v,r}^2,\phi_{w,t}) \dfrac{ \d r \, \d t \, \d u \, \d v \, \d w}{rt}\\
	&\quad+ \int\limits_{\substack{(u,v,w) \in \R^{3d} \\ r \ge t >0}} \lip f, \phi_{u,r} \rip \lip g, \psi_{v,r}^3\rip \lip h,\phi_{w,t} \rip \Lambda(\phi_{u,r},\psi_{v,r}^4,\phi_{w,t}) \dfrac{ \d r \, \d t \, \d u \, \d v \, \d w}{rt}. \\
	&=: I
	\end{align*}
Recall that only $\psi^1$ and $\psi^3$ are non-cancellative.
Split each of the three terms in the same way, obtaining a decomposition $\Lambda(f,g,h) = I + II + III$ where $I$ is given above while $II$ and $III$ are defined by
\begin{align*}
	II &= \int\limits_{\substack{(u,v,w) \in \R^{3d} \\ t \ge r >0}} \lip f,\phi_{u,r} \rip \lip g, \psi_{v,t}^1 \rip \lip h, \phi_{w,t} \rip \Lambda^{1*}(\phi_{w,t},\psi_{v,t}^2,\phi_{u,r}) \dfrac{ \d t \, \d r \, \d u \, \d v \, \d w}{tr}\\
	&\quad + \int\limits_{\substack{(u,v,w) \in \R^{3d} \\ t \ge r >0}} \lip f,\phi_{u,r} \rip \lip g, \psi_{v,t}^3 \rip \lip h, \phi_{w,t} \rip \Lambda^{1*}(\phi_{w,t},\psi_{v,t}^4,\phi_{u,r}) \dfrac{ \d t \, \d r \, \d u \, \d v \, \d w}{tr}\\
	&\quad+ \int\limits_{\substack{(u,v,w) \in \R^{3d} \\ t \ge s >0}} \lip f,\phi_{u,t} \rip \lip g,\phi_{v,s} \rip \lip h,\psi_{w,t}^1 \rip \Lambda^{1*}(\psi_{w,t}^2,\phi_{v,s},\phi_{u,t}) \dfrac{ \d t \, \d s \, \d u \, \d v \, \d w}{ts}\\
	&\quad + \int\limits_{\substack{(u,v,w) \in \R^{3d} \\ t \ge s >0}}\lip f,\phi_{u,t} \rip \lip g,\phi_{v,s} \rip \lip h,\psi_{w,t}^3 \rip \Lambda^{1*}(\psi_{w,t}^4,\phi_{v,s},\phi_{u,t}) \dfrac{ \d t \, \d s \, \d u \, \d v \, \d w}{ts}\end{align*}
	\begin{align*}
	III&= \int\limits_{\substack{(u,v,w) \in \R^{3d} \\ s \ge r >0}} \lip f,\psi_{u,r}^1 \rip \lip g,\phi_{v,s} \rip \lip h,\phi_{w,s} \rip \Lambda^{2*}(\psi_{w,s}^2,\phi_{v,s},\phi_{u,r}) \dfrac{ \d s \, \d r \, \d u \, \d v \, \d w}{sr}\\
	&\quad + \int\limits_{\substack{(u,v,w) \in \R^{3d} \\ s \ge r >0}}\lip f,\psi_{u,r}^3 \rip \lip g,\phi_{v,s} \rip \lip h,\phi_{w,s} \rip \Lambda^{2*}(\psi_{w,s}^4,\phi_{v,s},\phi_{u,r}) \dfrac{ \d s \, \d r \, \d u \, \d v \, \d w}{sr}\\
	&+ \int\limits_{\substack{(u,v,w) \in \R^{3d} \\ r \ge s >0}} \lip f,\phi_{u,r} \rip \lip g,\phi_{v,s} \rip \lip h,\psi_{w,r}^1 \rip \Lambda^{2*}(\phi_{u,r},\psi_{w,r}^2,\phi_{v,s}) \dfrac{ \d r \, \d s \, \d u \, \d v \, \d w}{rs}\\
	&+ \int\limits_{\substack{(u,v,w) \in \R^{3d} \\ r \ge s >0}} \lip f,\phi_{u,r} \rip \lip g,\phi_{v,s} \rip \lip h,\psi_{w,r}^3 \rip \Lambda^{2*}(\phi_{u,r},\psi_{w,r}^4,\phi_{v,s}) \dfrac{ \d r \, \d s \, \d u \, \d v \, \d w}{rs}.
\end{align*}
Due to the apparent symmetry, it is enough to handle only the first summand in $I$, let us call it $\sigma_1$. The remaining eleven terms are handled almost exactly the same.
Recalling $\Upsilon$ from (\ref{def:upsilon}) and the different regions of $Z^d$ from (\ref{eq:regions}), we can represent 
	\begin{equation}\label{eq:sigma0-4}
	\begin{aligned} \sigma_1 &= \int\limits_{(w,t) \in Z^d} \int\limits_{(u,v,s) \in Z(w,t)} \Upsilon(u,v,w,s,t) \lip f,\psi_{u,s}^1\rip \lip g,\phi_{v,s} \rip \lip h,\phi_{w,t} \rip \dfrac{ \d s \, \d u \, \d v \, \d t \, \d w}{st}\\
	& \quad + \left(\int\limits_{Z^d} \int\limits_{(u,v,s) \in Z(w,t)} - \int\limits_{Z^d} \int\limits_{(u,v,s) \in  Z(w,t) \backslash A(w,t)} \right)\Lambda(P_{k_1},\tilde P_{k_2},\phi_{w,t}) \\
		&\hspace{30ex} \times \lip f , \psi_{u,s}^1 \rip \lip g,\phi_{v,s} \rip \lip h,\phi_{w,t} \rip \dfrac{ \d s \, \d u \, \d v \, \d t \, \d w}{st}\\
	&\quad + \int\limits_{Z^d} \int\limits_{(u,v,s) \in A(w,t)} \left[ \Lambda(P_{k_1-1},\phi_{v,s}-\tilde P_{k_2}, \phi_{w,t}) + \Lambda(\psi^1-P_{k_1-1},\tilde P_{k_2}, \phi_{w,t})\right] \\
		& \hspace{30ex} \times \lip f,\psi_{u,s}^1\rip \lip g,\phi_{v,s} \rip \lip h,\phi_{w,t} \rip  \dfrac{ \d s \, \d u \, \d v \, \d t \, \d w}{st} \\
	&=\sigma_{1,0} + \sigma_{1,1}+\sigma_{1,2} + \sigma_{1,3}+\sigma_{1,4}.\end{aligned}
	\end{equation}
Therefore, using Proposition \ref{prop:psi-nu} and the kernel estimates on $\Upsilon$ (Lemma \ref{lemma:U}), we obtain $\nu_z \in \Psi_{z}^{k_1+k_2,\delta;1,0}$ such that
	\[ \sigma_{1,0} = \int_{Z^d} \lip f \otimes g, \nu_z \rip \lip h,\phi_z \rip \, \d \mu(z)=:U_{k_1+k_2}^1(f,g,h). \] 
Recalling the estimate (\ref{eq:half-est}), and again applying Proposition \ref{prop:psi-nu}, we obtain $\nu^j_z \in C\Psi_{z}^{j,\delta;1,0}$ such that
	\begin{align*} \sigma_{1,3} + \sigma_{1,4} &= \sum_{j=\min\{k_1,k_2\}}^{k_1+k_2-1} \int_{Z^d} \lip f \otimes g ,\nu^j_z \rip \lip h,\phi_z \rip \, \d\mu(z).
	\end{align*}
In this way we have constructed the remaining wavelet forms $U_{j}^1$ for $j=\min\{k_1,k_2\},\ldots,k_1+k_2-1$. The vanishing paraproducts assumption allows us to compute
	\[ \Lambda(P,\tilde P,\phi_{w,t}) = \sum_{|\gamma_i|=k_i}\lip \psi_{u,s}, \varu^{\gamma_1}_{w,t}\rip \lip \phi_{v,s}, \varu_{w,t}^{\gamma_2} \rip \lip b_{\gamma}^0 , (\partial^{-\gamma_1-\gamma_2} \phi)_{w,t} \rip. \]
By considering the supports of $\phi$, $\psi$, and $\varu^{\gamma_i}$, 
	\[ G(u,v,w,s,t) := \lip \phi_{v,s}, \varu^{\gamma_1}_{w,t} \rip \lip \psi_{u,s},\varu^{\gamma_2}_{w,t} \rip \]
vanishes whenever $|u-w| \ge 3\max\{s,t\}$ or $|v-w| \ge 3\max\{s,t\}$. Moreover, $|G| \lesssim \frac{s^M}{t^{2d+M}}$ for any $M$ up to which $\psi$ and $\phi$ have vanishing moments and remain smooth (see Lemma \ref{lemma:almost}). Therefore, for $\sigma_{1,2}$, the integration region $Z(w,t) \backslash A(w,t)$ can be replaced by $I(w,t)$ and by Lemma \ref{lemma:psi-theta}, there exists $\theta_z \in \Psi_{z}^{M,\delta;1,0}$ such that
	\[ \sigma_{1,2} = \sum_{|\gamma_\ell|=k_\ell} \int_{Z^d} \lip f \otimes g,\theta_z \rip \lip h,\phi_z \rip \lip b_\gamma,(\partial^{-\gamma_1-\gamma_2}\phi)_z \rip \, \d\mu(z).\]
Furthermore, since $b_\gamma^0 \in \BMO$ and $\partial^{-\gamma_1-\gamma_2} \phi$ has mean zero, each summand in $\sigma_{1,2}$ is a wavelet form with the wavelet
	\[ \nu_{z} = \lip b_\gamma,(\partial^{-\gamma_1-\gamma_2}\phi)_z \rip \theta_{z}. \]
We break up each $\sigma_i=\sigma_{i,0}+\sigma_{i,1}+\sigma_{i,2}+\sigma_{i,3}+\sigma_{1,4}$, $i=1,2,\ldots,12$. Each $\sigma_{i,0}$, $\sigma_{i,2}$ , $\sigma_{i,3}$, and $\sigma_{i,4}$ is handled similarly giving the wavelet forms $U_{j}^i$. 

We now deal with the remaining terms, $\sigma_{i,1}$. We reassemble $\sigma_{i,1}+\sigma_{i+1,1}+\sigma_{i+2,1}+\sigma_{i+3,1}$ for $i=1,5,9$ and use Lemma \ref{lemma:product} (the expanded tensor Calder\'on reproducing formula) to obtain
	\[ \sum_{\ell=1}^4\sigma_{\ell,1} = \sum_{|\gamma_i|=k_i}\int_{Z^d} \lip b_{\gamma},(\partial^{-\gamma_1-\gamma_2}\phi)_{z} \rip \lip f ,\varu_{z}^{\gamma_1} \rip \lip g , \varu_{z}^{\gamma_2} \rip \lip h, \phi_{z} \rip \, \d\mu(z)=: \sum_{|\gamma_\ell|=k_\ell} \Pi_{b_\gamma,\gamma}(f,g,h). \]
Doing so similarly for $\sum_{\ell=5}^8\sigma_{\ell,1}$ and $\sum_{\ell=9}^{12}\sigma_{\ell,1}$ and using the vanishing paraproducts assumption on $\Lambda^{1*}$ and $\Lambda^{2*}$ yields
	\begin{equation}\label{eq:rep-van} \Lambda(f,g,h) = \sum_{i=1}^6 \sum_{j=\min\{k_1,k_2\}}^{k_1+k_2} U_j^i(f,g,h) + \sum_{|\gamma_\ell|=k_\ell} \sum_{i=0}^2 \Pi^{i*}_{b_\gamma^{i*},\gamma}(f,g,h).\end{equation}
To remove the vanishing paraproducts assumption, we first recall the definition of $(k_1,k_2)$ paraproducts. This means that
	\[ \Lambda_{k_1,k_2} := \Lambda - \sum_{i=0}^2\sum_{ (|\gamma_1|,|\gamma_2|) < (k_1,k_2) } \Pi_{b_\gamma^{i*},\gamma}^{i*} \]
is a $(k_1,k_2,\delta)$ CZ form and has vanishing paraproducts of orders $<(k_1,k_2)$. Thus the Theorem is proved by applying (\ref{eq:rep-van}) to $\Lambda_{k_1,k_2}$.

\section{Sobolev Space Bounds}\label{sec:sob}
The cancellation structure of the forms $U$ and $\Pi$ is important for the  results below. To reflect this, we introduce the following intrinsic sub-trilinear form 
	\begin{equation}\label{eq:mss} \MSS(f,g,h) = \int_{Z^d} \Psi_z^{0,\delta;1,0} (f,g) \Psi^{\cals_0}_z (h) \, \d \mu(z) \end{equation}
where $\Psi_z^{\cals_0} f$ and $\Psi_z^{k,\delta;i,j} (f,g)$ are the intrinsic wavelet coefficients defined by
	\[\Psi_z^{0,\delta;i,j}(f,g) = \sup_{\nu_z \in \Psi_z^{0,\delta;i,j}} |\l f \otimes g,\nu_z \r|, \quad \Psi_z^{\cals_0} (f) = \sup_{\substack{\theta_z \in \Psi_z^{0,\delta;0}, \\ \|\Sy_z^{-1}\theta_z\|_{\star,2d+1,\delta} \le 1} } |\lip f,\theta_z \rip|. \]
It is important to distinguish among the three arguments, since the first one is non-cancellative, the first two have limited decay, and the third has rapid decay. In this first application, Section \ref{subsec:class-sob}, the collection of wavelets used in $\Psi^{\cals_0}$ will actually be compactly supported, but we will need to consider rapidly decaying ones in Section \ref{subsec:frac-sob}.

For the paraproducts, we define the intrinsic paraproduct form for $b \in \BMO$
	\begin{equation}\label{eq:pi} \pi_b(f,g,h) = \int_{Z^d} \Psi_z^{\cals_0} (b) \Psi_z^{0,\delta;1,1}(f,g) \Psi_z^{\cals_0} (h) \, \d\mu(z). \end{equation}
Estimates for $\Lambda$ are achieved using the representation theorem and then by appealing to estimates for $\MSS$ and $\pi_b$. In particular, these forms have sparse $(1,1,1)$ bounds which we will now define. 

\begin{definition}\label{def:sparse}
A collection $\Q$ of cubes $Q \subset \R^d$ is sparse if there is a disjoint collection of sets $\{E_Q : Q \in \Q\}$ such that
	\[ E_Q \subset Q \mbox{ and } |E_Q| > \frac 12 |Q|. \]
Above, $|\cdot|$ is the Lebesgue measure. A sub-trilinear form $\S$ has sparse $(p_1,p_2,p_3)$ bounds if for each triple $f_j \in L^\infty(\R^d)$  with compact support, $j=1,2,3$, there is a sparse collection $\Q = \Q(f_1,f_2,f_3)$ such that
	\begin{equation}
\label{e:sparsegen}
 \S(f,g,h) \le C \sum_{Q \in \Q} |Q| \lip f_1 \rip_{p_1,Q}\lip f_2\rip_{p_2,Q}\lip f_1 \rip_{p_3,Q}, \quad \lip f \rip_{p,Q} \coloneqq |Q|^{-1/p}\|f\cic{1}_Q\|_{p}.\end{equation}
\end{definition}
The fact that $\pi_b$ and $\MSS$ have sparse bounds can be achieved through standard approaches, see for example \cite{conde-alonso17,lacey-a2-17,lerner-a2-13,barron2017weighted}, since they are more or less Calder\'on-Zygmund forms. However, in Section \ref{sec:sparse}, Proposition \ref{prop:sparse} below, a direct proof is given. Such proof also applies to more general forms which do not necessarily satisfy  kernel estimates.
\subsection{Weight classes} \label{ss:wclasses}
Sparse bounds are naturally related to weighted norm  inequalities. Accordingly,  the definition of the multilinear Muckenhoupt $A_{\vec p,\vec r}$ weights, first appearing in \cite{lerner2009new}, is recalled below. We choose to employ the normalization of \cite[pp.\ 101-102]{li2020extrapolation} and stick to trilinear weight vectors, but the extension to higher linearities is a mere matter of changing the notation.

Throughout this discussion, unless otherwise specified, a weight vector $\vec v=(v_1,v_2,v_3)$ refers to a triple of positive measurable functions on $\R^3$ such that
\begin{equation}
\label{e:weight1} 1= \prod_{j=1}^3 v_j(x), \qquad x\in \R^d.
\end{equation}
For $\vec \eps=(\eps_1,\eps_2,\eps_3) \in (0,\infty]^3$, and a  weight vector $\vec v$, define the characteristic
\[
[\vec v]_{\vec \eps} =  \sup_{Q} \prod_{j=1}^d \left\langle \frac{1}{v_j} \right\rangle_{\eps_j,Q}
\]
where $Q$ is allowed to vary among all cubes in $\R^d$. Notice that if $\eps_j=\infty$ the corresponding local norm simply indicates the essential supremum on $Q$. We work with the extended simplex $S$ and with the  set of generalized H\"older tuples $P$
\[
S=\left\{\vec \alpha  \in  \left[{\textstyle -\frac{1}{2}},  1\right]^3: \sum_{j=1}^3 \alpha_j=1\right\}, 
\quad
P=\left\{\vec p=(p_1,p_2,p_3)\in (-\infty,\infty]^3: \textstyle\left(\frac{1}{p_1},\frac{1}{p_2}, \frac{1}{p_3}\right) \in S \right\}.\] 
Say that the tuple $\vec r=(r_1,r_2,r_3)\in [1,\infty)^3$ satisfies $\vec r \prec \vec p$  for $\vec p \in P$, if 
\[
\eps_{j}\coloneqq \frac{p_jr_j}{p_j-r_j} >0, \qquad 
j=1,2,3.
\]
Above, we mean that $\eps_{j}=r_j$ if $p_j=\infty$, in natural agreement with taking limits in the definitions. 
If  $\vec r \prec \vec p$, writing $\vec \eps(\vec p, \vec r)$ for  $\vec \eps$ defined above, the weight vector  $\vec v$ belongs to the class $A_{\vec p ,\vec r}$ if
$
[\vec v]_{A_{\vec p ,\vec r}} \coloneqq
[\vec v]_{\vec \eps(\vec p, \vec r)}  <\infty.
$
This definition, unlike that of \cite{li2020extrapolation}, is completely symmetric with respect to matching permutations of  $\vec v, \vec p, \vec r$.
However our purpose of studying bilinear operators acting on Lebesgue and Sobolev spaces whose integrability exponents $p_1,p_2$ are $\geq 1$, it is convenient to break the symmetry and work in the corresponding  portion of $P$. To wit, define $$
P_\circ=\left\{\vec p \in P: 1<\min\{ p_1,p_2\}< \infty \right\}, \qquad p(\vec p) \coloneqq \frac{p_3}{p_3-1} =\frac{p_1p_2}{p_1+p_2}. 
$$
Notice that $\frac12< p(\vec p) <\infty$ is automatic from the definition of $P_\circ$. On the other hand, also observe that (at most) one of $p_1,p_2$ may be $=\infty$ when $\vec p \in P_\circ$. 
For comparison with  \cite[pp.\ 101-102]{li2020extrapolation}, when $\vec v\in A_{\vec p ,\vec r}$, one may   single out the dual weight
\[
w=w( \vec v)= \prod_{j=1}^2 v_j = \frac{1}{v_3}
\]
corresponding to the weight $w$ associated to the pair $ (v_1,v_2)$ therein. We will not make use of the notation $w( \vec v)$ in our statements to minimize redundancy.
The most important classes for the study of bilinear Calder\'on-Zygmund operators correspond to the choice $r_1=r_2=r_3=1$. In that case, we simply write $[\vec v]_{A_{\vec p}}$ in place of   $[\vec v]_{A_{\vec p ,\vec r}}$.

A consequence of the sparse bounds of Proposition \ref{prop:sparse} below for $\pi_b$ and $\MSS$ is the following weighted Lebesgue space result.
\begin{proposition}\label{prop:weighted}
{Let $\vec p\in P_\circ$, $\vec v   \in A_{\vec p}$, $p=p(\vec p)$, $b \in \BMO$. Then, denoting by $T$ the bilinear operator defined by either $\lip T(f_1,f_2),h\rip = \pi_b(f_1,f_2,h)$ or $\MSS(f_1,f_2,h)$, there holds
	\[ \left\|\frac{T(f_1,f_2)}{ v_3}\right\|_{L^p(\R^d)} \lesssim [\vec v]_{A_{\vec p}}^{\max\{p_1',p_2',p\}} \prod_{j=1}^2\left\|f_j v_j \right\|_{L^{p_j}(\R^d)}. \]}
\end{proposition}
\begin{proof} Let $\vec q= (3,3,3)$. Fix $f_1,f_2$ and $h$.  Proposition \ref{prop:sparse} yields the existence of a sparse collection $\Q$ such that  
\[
\begin{split}
\left|\left \langle T(f_1,f_2)v_3^{-1},h  \rip \right\rangle \right| \lesssim 
\sum_{Q \in \Q} |Q|   \lip  f_1\rip_{1,Q}\lip f_2\rip_{1,Q}\lip h v_3^{-1} \rip_{1,Q}
\lesssim  [\vec v]_{A_{\vec q}}^{\frac32} \|f_1 v_1\|_3 \|f_2 v_2\|_3 \|h\|_3.
\end{split}
\]   The sharp $A_{\vec q}$ weighted norm inequality for the sparse forms,   see \cite[Lemma 6.1]{culiuc2018domination}, has been used for the last bound.  The \emph{a priori} estimate we obtained is in particular the case $\vec p = (3,3,3)$ of the Proposition. The general case is now obtained by invoking the extrapolation result \cite[Theorem 2.1]{li2020extrapolation}.
\end{proof}

\subsection{Classical Weighted Sobolev Spaces}\label{subsec:class-sob}
{Recall that  the weighted Sobolev norm $W^{k,p}(v)$, for $k \in \N$, $0<p\le \infty$, and a weight $v$, is given by
	\[ \|f\|_{\dot W^{k,p}(v)} = \sum_{|\alpha|=k} \|[\partial^\alpha f]v\|_{L^p(\R^d)}, \quad \|f\|_{W^{k,p}(v)} = \sum_{j=0}^k \|f\|_{\dot W^{j,p}(v)}. \]}
\begin{theorem}\label{thm:sob-min}
Let $k_0,k_1,k_2 \in N$ with $k_0 \le \min\{k_1,k_2\}$ and $\delta>0$. Let $\Lambda$ be a $(k_1,k_2,\delta)$ CZ form such that 
	\begin{equation}\label{eq:b0} D^{k_0-|\gamma|}b^0_\gamma \in \BMO, \quad |\gamma| < k_0. \end{equation}
Then, \begin{equation}\label{eq:dom-sob}
	\begin{aligned} \sum_{|\kappa|=k_0}\Lambda(f,g,\partial^\kappa h) &\lesssim \sum_{|\gamma|=k_0} \MSS(f,\partial^{\gamma}g,h) + \MSS(g,\partial^\gamma f,h)\\ 
		&\quad + \sum_{|\gamma_\ell| \le k_\ell} \sum_{|\beta|=\max\{0,k_0-|\gamma|\}} \sum_{|\alpha|=k_0-|\beta|}  \pi_{\partial^\beta b_\gamma^0}(\partial^{\alpha_1} f, \partial^{\alpha_2}g,h)  \\
		&\quad + \sum_{|\gamma_\ell| \le k_\ell} \sum_{|\beta|=\max\{0,k_0-|\gamma|\}} \sum_{|\alpha|=k_0-|\beta|}  \pi_{\partial^\beta b_\gamma^0}(\partial^{\alpha_1} f, \partial^{\alpha_2}g,h)  \\
		&\quad+ \sum_{|\gamma_\ell| \le k_{\ell}}  \sum_{|\alpha|=k_0} \pi_{b_\gamma^{1}}(h,g,\partial^{\alpha} f) + \pi_{b_\gamma^{2}}(f,h,\partial^{\alpha}g). 
	\end{aligned} \end{equation}{ 
Proposition \ref{prop:weighted} then leads to  the following bounds. For $\vec p\in P_\circ$, $\vec v   \in A_{\vec p}$, $p=p(\vec p)$,	\begin{align} \label{eq:sob-hom} \|T(f,g)\|_{{\dot W}^{k_0,p}\left(\frac{1}{v_3}\right)} &\lesssim [\vec v]_{A_{\vec p}}^{\max\{p_1',p_2',p\}} \sum_{0 \le i+j \le k_0} \|f\|_{\dot W^{i,p_1}(v_1)} \|g\|_{\dot W^{j,p_2}(v_2)},\\
		\label{eq:sob-inhom}\|T(f,g)\|_{W^{k_0,p}\left(\frac{1}{v_3}\right)} &\lesssim [\vec v]_{A_{\vec p}}^{\max\{p_1',p_2',p\}} \|f\|_{W^{k_0,p_1}(v_1)} \|g\|_{W^{k_0,p_2}(v_2)}. \end{align}}
\end{theorem}

If $T$ has vanishing paraproducts, i.e. $b_\gamma^0=0$, for $(|\gamma_1|,|\gamma_2|)<(k_1,k_2)$, then the corresponding terms vanish from (\ref{eq:dom-sob}) so that (\ref{eq:sob-hom}) becomes 
	\[ \|T(f,g)\|_{{\dot W}^{k_0,p}\left(\frac{1}{v_3}\right)} \lesssim [\vec v]_{A_{\vec p}}^{\max\{p_1',p_2',p\}} \left( \|f\|_{\dot W^{k_0,p_1}(v_1)} { \|gv_2\|_{L^{p_2}(\R^d)}} +  {\|fv_1\|_{L^{p_1}(\R^d)}} \|g\|_{\dot W^{k_0,p_2}(v_2)} \right). \]

Theorem \ref{thm:sob-min} is sharp in a couple of ways, but the precise sense must be explained. First, the appearance of the many norms on the right hand side is necessary when considering the entire class of $(k_1,k_2,\delta)$ CZOs. In other words, for each pair $(i,j)$ with $0 \le i+j \le k_0$, we can exhibit a $(k_1,k_2,\delta)$ CZO which only maps $\dot W^{i,p_1} \times \dot W^{j,p_2}$ into $\dot W^{k_0,p}$. In fact, such an operator is 
	\[ \l T(f,g),h \r = \Pi_{b,\gamma}(f,g,h), \quad |\gamma_1|=i, \, |\gamma_2|=j, \,D^{k_0-|\gamma|}b \in \BMO.\]
Taken in a similar sense, both the exponent of the weight characteristic and the condition $D^{k_0-|\gamma|}b_\gamma^0 \in \BMO$ are sharp. Concerning the exponent,   the following sharpness result holds.
\begin{proposition}\label{prop:exp}
Fix $\vec p\in P_\circ$,  $k_0,k_1,k_2 \in \N$ with $k_0 \le \min\{k_1,k_2\}$. For any weight vector $\vec v$, define
 \[ \|T\|_{k_0,\vec p,\vec v} \coloneqq \sup_{f,g}\frac{\|T(f,g)\|_{\dot W^{k_0,p}\left(\frac{1}{v_3}\right)}}{\sum\limits_{0 \le i+j \le k_0} \|f\|_{\dot W^{i,p_1}(v_1)} \|g\|_{\dot W^{j,p_2}(v_2)}}. \]
Then, for each $M,\delta>0$ and $\varphi$ satisfying $\varphi(t) = o(t^{\max\{p_1',p_2',p\}})$ as $t \to \infty$,
	\[ \sup_{\vec v, \, T} \frac{\|T\|_{\vec p,k_0,\vec v}}{\varphi([\vec v]_{A_{\vec p}})} = \infty\]
where the supremum is taken over all $\vec v \in A_{\vec p}$ and all $T$ which are $(k_1,k_2,\delta)$ CZOs satisfying $\|T\|_{\vec p,k_0,\vec w} \le M$ with $\vec w=(\cic{1}_{\R^d},\cic{1}_{\R^d}, \cic{1}_{\R^d}) $. 
\end{proposition}
The analogous statement switching the input spaces also holds. When $k_0=0$, this result is contained in \cite[pp. 763-764]{li2014sharp}. Below, in Section \ref{subsubsec:sharp-proof}, we will give a slight modification of their proof, adapted to our smooth operators.

As regards the paraproduct assumptions (\ref{eq:b0}), notice that we are imposing additional requirements on the derivatives of $b_\gamma^0$, but not $b_\gamma^1$ and $b_\gamma^2$. One may compare to the linear case, where one can obtain $T:\dot W^{k,p}\to \dot W^{k,p}$ if and only if $b_0^\gamma$ \textit{vanishes} for $|\gamma| < k_0$. This phenomenon persists in the multilinear setting in the following sense.

\begin{proposition}\label{prop:bmo}
Let $T$ be a $(j+i,\delta)$ SI operator satisfying
	\[ \|T\|_{\dot{W}^{j,p_1} \times \dot{W}^{i,p_2} \to \dot W^{k,p}} \lesssim 1 \] 
for some $\vec p \in P_\circ$ and $p=p(\vec p) \in (d, \infty)$. Then for $(|\gamma_1|,|\gamma_2|) \le (j,i)$,
	\[ D^{k} T(x_1^{\gamma_1},x_2^{\gamma_2}) \left\{ \begin{array}{cl} \in BMO & \mbox{ for } |\gamma_1| = j, |\gamma_2|=i; \\
			=0 & \mbox{ for } |\gamma| < j+i. \end{array} \right. \]
\end{proposition}

As in the above discussion, such a result does not apply to our multilinear operators directly since they do not, as a whole, preserve homogeneous Sobolev spaces. However, if some portion of the operator---as seen from the proof, $\Pi_{b_\gamma^0}$---does map accordingly, then we recover the conditions on $b_\gamma^0$. We present the proof of Proposition \ref{prop:bmo} in Section \ref{subsubsec:sharp-proof}.

\subsubsection{Proof of Theorem \ref{thm:sob-min}}\label{subsubsec:proof-sob-min}
We first deal with $U_j^1$ for $j \ge k_0$.
Integrating by parts twice, we have, with $z=(w,t)$,
	\[ \sum_{|\kappa|=k_0} \lip f \otimes g, \nu_z \rip \lip \partial^\kappa h,\phi_z \rip  = \sum_{|\kappa|=k_0} \lip f \otimes g, \nu_z \rip t^{-k_0} t^{k_0}\lip h,\partial^\kappa\phi_z \rip \]
	\[= \sum_{|\kappa|=k_0} \sum_{|\gamma|=k_0}\lip f \otimes \partial^{\gamma}g, t^{-k_0}\partial_y^{-\gamma}\nu_z\rip \lip h,t^{k_0} \partial^\kappa\phi_z \rip. \]
Noting that $t^{|\gamma|} \partial^{-\gamma}_y \Psi^{j,\delta;1,0}_{z} \subset \Psi^{0,\delta;1,0}_{z}$ for $|\gamma| \le j$ (see Proposition \ref{prop:psi-nu}), we obtain $\nu'_z \in \Psi^{0,\delta;1,0}_{z}$ and $\phi'=\partial^{\kappa} \phi$ with mean zero such that
	\[ \sum_{|\kappa|=k_0} \lip f \otimes g, \nu_z \rip \lip \partial^\kappa h,\phi_z \rip = \sum_{|\kappa|=k_0}\sum_{|\gamma|=k_0}\lip f \otimes \partial^{\gamma} g, \nu'_z \rip \lip h,\phi'_z \rip.\]
In this way, 
	\[ \sum_{|\kappa|=k_0}\left| U(f,g,\partial^{\kappa} h) \right| \lesssim \sum_{|\gamma|=k_0}\MSS(f,\partial^{\gamma}g,h).\]
The same argument is applied to each $U_j^2$, placing all the derivatives on $f$ and bounding above by $\sum_{|\gamma|=k_0}\Pi(g,\partial^\gamma f,h)$ since in this case $g$ is in the non-cancellative position. For the remaining wavelet forms, $U_j^3,\ldots,U_j^6$---the ``adjoint'' ones---we have more freedom with the derivatives. In fact, for any $|\kappa| =k^* \le k_1+k_2$, we can unwind the wavelet form and integrate by parts there. Let us only do $U_3^j$. Integrating by parts twice as before,
	\[ U_j^3(f,g,\partial^\kappa h) = \int\limits_{\substack{(w,t) \in Z^d\\(u,v,s) \in Z(w,t)}} \Upsilon_j(u,v,w,s,t) \lip f,\phi_{w,t} \rip \lip g,\psi_{v,s}^1 \rip \lip \partial^{\kappa} h,\phi_{u,s} \rip \dfrac{\d s \, \d u \, \d v \, \d t \, \d w}{st}\]	
	\[ = \sum_{|\gamma|=k^*} \int\limits_{\substack{(w,t) \in Z^d\\(u,v,s) \in Z(w,t)}} \frac{\Upsilon_j(u,v,w,s,t) t^{|\kappa|}}{s^{|\kappa|}} \lip \partial^{\gamma}f,(\partial^{-\gamma}\phi)_{w,t} \rip \lip g,\psi_{v,s}^1 \rip \lip h,(\partial^\kappa\phi)_{u,s} \rip \dfrac{\d s \, \d u \, \d v \, \d t \, \d w}{st}.\]
The new symbol $\Upsilon_j(u,v,w,s,t)(\frac{t}{s})^{|\gamma|}$ satisfies much better estimates than needed in Proposition \ref{prop:psi-nu}, and in fact supplies $\nu'_z \in \Psi^{j+|\gamma|,\delta;0,1}_{z} \subset \Psi^{0,\delta;0,1}_{z}$ which by the same argument as before gives $\sum_{|\kappa|=k^*} |U_j^i(f,g,\partial^\kappa h)| \lesssim \sum_{|\gamma|=k^*} \MPS(g,\partial^{\gamma}f,h)$.

The argument used on $U_j^1,U_j^2$ also applies to $\Pi^{i*}_{b_\gamma^i,\gamma}$ for $i=1,2$ 
since they have the same cancellation structure in the first two arguments. In this way, 
	\[ \sum_{i=1}^2 \sum_{|\kappa|=k^*} \Pi^{i*}_{b_\gamma^i,\gamma}(f,g,\partial^\kappa h) \lesssim \sum_{|\gamma|=k^*} \pi_{b_\gamma^1}(h,g,\partial^{\gamma}f) + \pi_{b_\gamma^2}(f,h,\partial^{\gamma}g) \]
and for $\tilde \Pi$.
However, we can actually see that the number of derivatives, $k^*$, can be taken all the way up to $k_1+k_2$ since we placed the extra ones in the fully cancellative position. 

The final term to estimate is the paraproduct $\Pi_{b_\gamma^0,\gamma}$. We are restricted here since the $\gamma_\ell$-family $\varu^{\gamma_\ell}_z$ only has vanishing moments up to $\gamma_\ell$. So we can only place $|\gamma_1|$ derivatives on $f$ and $|\gamma_2|$ on $g$, and $k^*-|\gamma|$ derivatives remain on $h$. These must go on the symbol $b_\gamma$. We follow the above reasoning to obtain 
	\begin{align*} \sum_{|\kappa|=k^*} \Pi_{b_\gamma,\gamma}(f,g,\partial^{\kappa}h) &= \sum_{|\kappa|=k^*} \sum_{|\alpha_\ell|=|\gamma_\ell|} \sum_{|\beta|=k^*-|\gamma|} \int \lip \partial^{\beta} b, (\partial^{-\beta-\gamma_1-\gamma_2} \phi)_{w,t} \rip \dfrac{\d t\d w}{t}\\
		&\quad \times \lip \partial^{\alpha_1} f, (\partial^{-\alpha_1}\varu^{\gamma_2})_{w,t} \rip \lip \partial^{\alpha_2} g, (\partial^{-\alpha_2}\varu^{\gamma_2})_{w,t} \rip \lip h,(\partial^{\kappa} \phi)_{w,t} \rip \\
	&\le \sum_{|\alpha_\ell|=|\gamma_\ell|} \sum_{|\beta|=k^*-|\gamma|} \pi_{\partial^\beta b}(\partial^{\alpha_1}f,\partial^{\alpha_2}g,h). 
	\end{align*}
\qed

From the proof we can see that the only thing holding us back from taking $k_0$ all the way up to $k_1+k_2$ (represented by $k^*$ in the proof) are the wavelet forms $U_j^1$ and $U_j^2$ for $j< k_1+k_2$. Thus, if these vanish or have some hidden regularity, we obtain the better result.

\begin{cor}\label{cor:sob-van}
Let $\Lambda$ be a $(k_1,k_2,\delta)$ CZ form. Assume that $k^* \le k_1+k_2$, $D^{k^*-|\gamma|}b_\gamma^0 \in \BMO$ for $|\gamma| \le k^*$, and that $U_j^1$ and $U^2_j$ are zero for $j < k^*$. Then, for any {$\vec p \in P_\circ$, $p=p(\vec p)$, and $\vec v \in A_{\vec p}$},
	\[ \|T(f,g)\|_{{W}^{k^*,p}(\R^d,\frac 1 {v_3})} \lesssim [\vec v]_{A_{\vec p}}^{\max\{p_1',p_2',p\}} \|f\|_{W^{k^*,p_1}(v_1)} \|g\|_{W^{k^*,p_2}(v_2)} .\]
\end{cor}

The assumption that $U_j^1$ and $U_j^2$ vanish is the same as checking that $\Lambda$ has \textit{vanishing half paraproducts} of appropriate orders. However, this condition can be difficult to check, and it is not even clear that the product operator satisfies this assumption. One could replace this with the assumption that $U_j^{1},U_j^2$ are smoothing is some appropriate sense, similar to the assumption that $D^{k^*-|\gamma|}b_\gamma \in \BMO$, which makes the paraproduct forms smoothing operators. This is morally what we do in assuming $\Lambda$ has \textit{more} than $k_0$ smoothness in Theorem \ref{thm:sob-min}.

\subsubsection{Proofs of Sharpness}\label{subsubsec:sharp-proof}
Let us now return to the proofs of Propositions \ref{prop:exp} and \ref{prop:bmo} which demonstrate the sharpness of the exponent on the weight characteristic and the paraproduct assumptions (\ref{eq:b0}) in Theorem \ref{thm:sob-min}.
\begin{proof}[Proof of Proposition \ref{prop:exp}] From \cite{li2014sharp}, the bilinear operator given by 
	\[ R(f,g)(x_0) = \int K(x_0-x_1,x_0-x_2)f(x_1)g(x_2), \quad K(y_1,y_2) = \frac{y_1^1y_2^1}{(|y_1|^2+|y_2|^2)^{(2d+1)/2}} \]
satisfies
	\[ \sup_{\vec v \in A_{\vec p}} \frac{\|R\|_{0,\vec p,\vec v}}{\varphi([\vec v]_{A_{\vec p}})} = \infty. \]
$R$ is a $(k_1,k_2,1)$ CZO for any $k_1,k_2 \in \N$ and for any $\alpha$,
	\[ \partial^\alpha R(f,g) = \sum_{\gamma \le \alpha} R(\partial^\gamma f,\partial^{\alpha-\gamma} g) \]
which implies that $\|R\|_{k_0,\vec p,\vec v} \sim \|R\|_{0,\vec p,\vec v}$ for any weight $\vec v$.
\end{proof}

\begin{proof}[Proof of Proposition \ref{prop:bmo}]
Let $|\alpha|=k$, $|\gamma_1|=j$, and $|\gamma_2|=i$.
First we will show for each cube $Q$, setting $p_Q(x_1,x_2) = (x_1-x_Q)^{\gamma_1}(x_2-x_Q)^{\gamma_2}=:p_Q^{\gamma_1}(x_1)p_Q^{\gamma_2}(x_2)$,
	\begin{equation}\label{eq:tq} \fint_Q |\partial^{\alpha}Tp_Q-c_Q| \d x\lesssim 1 \end{equation}
where $c_Q = \int \partial^{\alpha}_{x_0} K(x_1,x_2,x_Q)(1-\phi_Q)p_Q(x_1,x_2) \, \d x_1 \, \d x_2$ and $x_Q$ is the center of the cube $Q$. Break up $p_Q = \phi_Qp_Q + (1-\phi_Q) p_Q$ where $\phi_Q(x_1,x_2) = \phi(\frac{x_1-x_Q}{\ell(Q)})\phi(\frac{x_2-x_Q}{\ell(Q)})$ and $\phi \in C^\infty_0(B(0,2))$ with $\phi =1$ on $B(0,1)$. We estimate
	\[ \fint_Q |\partial^{\alpha}T\phi_Qp_1|\d x \le |Q|^{-1} |Q|^{1-1/p} \|D^j \phi_Q p_Q^{\gamma_1}\|_{L^{p_1}(Q)} \|D^i \phi_Qp_Q^{\gamma_2}\|_{L^{p_2}(Q)}\lesssim 1. \]
On the other hand, for $x_0 \in Q$,
	\[ \begin{aligned}
	\Bigr|\int_{Q^c}\partial^\alpha_{x_0} K(x_1, & x_2,x_0)-\partial^{\alpha}_{x_0}K(x_1,x_2,x_Q)(1-\phi_Q(x_1,x_2))(x_1-x_Q)^{\gamma_1}(x_2-x_Q)^{\gamma_2} \, \d x_1 \, \d x_2 \Bigr|  \\
		&\lesssim \ell(Q)^\delta \int_{Q^c} \frac{|x_1-x_Q|^j |x_2-x_Q|^i}{(|x_Q-x_1|+|x_Q-x_2|)^{2d+j+i+\delta}} \, \d x_1 \, \d x_2\lesssim 1 .
	\end{aligned} \]
Thus (\ref{eq:tq}) is established. To complete the proposition, we must replace $Tp_Q$ by $T$ applied to the untranslated polynomial. To do so, we will show that the mapping properties of $T$ imply that $T(x_1^{\kappa_1},x_2^{\kappa_2})$ is a polynomial of degree $<k$ for $(|\kappa_1|,|\kappa_2|) < (j,i)$. Indeed, $\|x_1^{\kappa_1}\phi(x_1R^{-1})\|_{\dot{W}^{j,p}} \lesssim R^{|\kappa_1|-j+d/p_1}$ and similarly for $\kappa_2$, $i$,  and $p_2$. Thus, for any $\psi \in \cals_{k}$,
	\[ \begin{aligned}|\lip T(x_1^{\kappa_1},x_2^{\kappa_2}), \psi \rip| &= \lim_{R \to \infty} |\lip T(x_1^{\kappa_1}\phi(x_1R^{-1}),x_2^{\kappa_2}\phi(x_2R^{-1})),\psi \rip| \\
		&\lesssim \lim_{R\to \infty}\|T(x_1^{\kappa_1}\phi(x_1R^{-1}),x_2^{\kappa_2}\phi(x_2R^{-1}))\|_{\dot W^{k,p}} \\
		&\lesssim \lim_{R\to\infty} \|x_1^{\kappa_1}\phi(x_1R^{-1})\|_{\dot{W}^{j,p_1}} \|x_2^{\kappa_2}\phi(x_2R^{-1})\|_{\dot{W}^{i,p_2}}=0 \end{aligned} \]
which implies $T(x_1^{\kappa_1},x_2^{\kappa_2})$ is a polynomial of degree strictly less than $k$. In particular, \[\partial^\alpha T(x_1^{\kappa_1},x_2^{\kappa_2})=0.\]
For each cube $Q$ and $(|\gamma_1|,|\gamma_2|)=(j,i)$,
	\[ T(x_1^{\gamma_1},x_2^{\gamma_2}) = \sum_{\kappa < \gamma} c_{\kappa,\gamma} x_Q^{\gamma-\kappa} T(x_1^{\kappa_1},x_2^{\kappa_2}) + Tp_Q. \] 
Therefore $\partial^\alpha T(x_1^{\gamma_1},x_2^{\gamma_2}) = \partial^\alpha Tp_Q$.
\end{proof}

\subsection{Weighted Fractional Sobolev Space Estimates}\label{subsec:frac-sob}
Until now $D^\sigma$, the Fourier multiplier by $|\xi|^\sigma$, has only been used when $\sigma$ is a positive integer. We now generalize to any $\sigma$ positive. Denote by $W^{\sigma,p}(v)$ the weighted inhomogeneous fractional Sobolev space on $\R^d$ with norm defined by
	\[ \|f\|_{W^{\sigma,p}(v)} = \|[D^\sigma f] v\|_{L^p(\R^d)} + \sum_{j=0}^{\lfloor \sigma \rfloor} \|[D^{j}f]v\|_{L^p(\R^d)}. \]

Two new features enter here, which will require some modification of the intrinsic forms. First, $D^\sigma$ applied to a mother wavelet will no longer be a mother wavelet, however it (and many of its derivatives) will still have rapid decay. Second, when $D^\sigma$ is applied to a non-cancellative wavelet, its decay will only be $\l \cdot \r^{d+\sigma}$ (see Lemma \ref{lemma:frac-wave} below). This motivates the introduction of limited decay wavelets and intrinsic forms, $\pi_b^\sigma$ and $\MSSsig$, which we will define later. For now, know that they are defined the same as $\pi_b$ and $\MSS$ above, but the wavelet class corresponding to the first argument has worse decay, depending on $\sigma$. Furthermore, as we will show in Proposition \ref{prop:sparse} below, these forms have sparse $(1,p_2,p_3)$ bounds for 
	\[ \frac{1}{p_2} + \frac{1}{p_3} < \frac{\sigma+d}{d}. \]
This implies the following weighted Lebesgue space bounds.

\begin{proposition}\label{prop:weighted-sigma}
Let $T$ be the bilinear operator defined by either $\lip T(f_1,f_2),h\rip = \pi_b^\sigma(h,f_1,f_2)$ or $\MSSsig(h,f_1,f_2)$, defined by (\ref{eq:msssig}) and (\ref{eq:pi-sig}) with $b \in \BMO$. Let $\vec p\in P_\circ$, $\vec v   \in A_{\vec p}$, $p=p(\vec p)>\frac{d}{\sigma+d}$. For $\vec r = (r_1,r_2,1)$ with $1 \le r_i < p_i$ and $\frac 1{r_1} + \frac 1{r_2} < \frac{\sigma+d}{d}$,
	\[ \left\|\frac{T(f_1,f_2)}{ v_3}\right\|_{L^p(\R^d)} \lesssim [\vec v]_{A_{\vec p,\vec r}}^{\max\{\frac{p_i}{p_i-r_i},p\}} \prod_{j=1}^2\left\|f_j v_j \right\|_{L^{p_j}(\R^d)}. \]
\end{proposition}

This proposition is proved in the same manner as Proposition \ref{prop:weighted}.

\begin{theorem}\label{thm:sob-frac}
Let $k_1,k_2 \in \N$ and $\sigma,\delta>0$ with $\sigma \le \min\{k_1,k_2\}$. If $\Lambda$ is a $(k_1,k_2,\delta)$ CZ form satisfying 
	\[ D^{\sigma-|\gamma|}b^0_\gamma \in \BMO, \quad |\gamma| < \sigma,\]
then, 
	\begin{equation}\label{eq:frac-dom} \begin{aligned} &\quad \left| \Lambda(f,g,D^\sigma h) \right| \lesssim \MSSsig(h, g,D^\sigma f) + \MSSsig(h,f,D^\sigma g) + \MSS(f, D^\sigma g,h) + \MSS(g,D^\sigma f,h)\\ 
		& + \sum_{|\gamma_\ell| \le k_\ell} \pi^\sigma_{b_\gamma^{1}} (h,g,D^\sigma f) + \pi_{b_\gamma^{2}}^\sigma(h,f,D^\sigma g) \\
		& +\sum_{\substack{|\gamma_\ell| \le k_\ell,\\ |\gamma| < \sigma }} \pi_{D^{\sigma-|\gamma|} b^0_\gamma}(D^{|\gamma_1|} f, D^{|\gamma_2|}g,h) + \sum_{\substack{|\gamma_\ell| \le k_\ell, \\ |\gamma|>\sigma}} \pi_{b_\gamma^0}(D^{\min\{\sigma,|\gamma_1|\}} f, D^{\sigma-\min\{\sigma,|\gamma_1\}}g,h). 
	\end{aligned} \end{equation}
Applying the estimates on $\MSSsig$, $\MSS$, $\pi_b^\sigma$, and $\pi_b$ from Propositions \ref{prop:weighted} and \ref{prop:weighted-sigma}, we obtain the following fractional weighted Sobolev space estimate. Let $\vec p \in P_\circ$, $p=p(\vec p) > \frac{\sigma+d}{d}$, and $\vec r = (r_1,r_2,1)$ satisfying $1 \le r_i < p_i$ and $\frac{1}{r_1}+\frac{1}{r_2} < \frac{\sigma+d}{d}$. Then, for any $\vec v \in A_{\vec p,\vec r}$,
	\[ \|T(f,g)\|_{W^{\sigma,p}\left(\frac{1}{v_3}\right)} \lesssim [\vec v]_{A_{\vec p,\vec r }}^{\max\{\frac{p_i}{p_i-r_i},p\}}\|f\|_{W^{\sigma,p_1}(v_1)} \|g\|_{W^{\sigma,p_2}(v_2)}. \]
\end{theorem}

As in the classical case, if $b_\gamma^0=0$ then the corresponding terms vanish in (\ref{eq:frac-dom}) and we obtain the following simplification for the homogeneous norms.
\begin{cor}\label{cor:leibniz}
If in addition to the assumptions of Theorem \ref{thm:sob-frac}, 
	\[ b_\gamma^0=0, \quad (|\gamma_1|,|\gamma_2|) < (k_1,k_2),\]
then for any $\vec v \in A_{\vec p,\vec r}$,
	\[ \begin{split}
	&\quad \left\|\frac{D^\sigma T(f,g)}{v_3}\right\|_{L^p(\R^d)} \\ & \lesssim [\vec v]_{A_{\vec p,\vec r}}^{\max\{\frac{p_i}{p_i-r_i},p\}} \left( \|[D^\sigma f]v_1\|_{L^{p_1}(\R^d)}\|gv_2\|_{L^{p_2}(\R^d)} + \|fv_1\|_{L^{p_1}(\R^d)}\|[D^\sigma g]v_2\|_{L^{p_2}(\R^d)} \right). \end{split}\]
\end{cor}

These new forms will only be needed to represent $U_j^3$, $U_j^6$, and $\Pi_{b_\gamma^i,\gamma}^{i*}$ for $i=1,2$---the ones with $h$ in a non-cancellative position. So first we give the proof representing all the other terms using the original intrinsic forms $\MSS$ and $\pi_b$.

\subsubsection{Proof of terms involving $\MSS$ and $\pi_b$}
For the remaining wavelet forms, $U_j^1,U_j^2,U_j^4,U_j^5$, we use the fact that $D^{\pm \sigma} \phi_{u,s} =s^{\mp\sigma}(D^{\pm\sigma} \phi)_{u,s}$ and, since $\widehat \phi$ has a zero of order larger than $\sigma$ at the origin, $\|D^{\pm\sigma} \phi\|_{\star,M,1} \lesssim 1$ for any $M>0$. Let us focus on $U_j^1$. For any $\sigma \le \min\{k_1,k_2\}$, we unwind the wavelet form
	\begin{multline*} U^1_j(f,g,D^\sigma h) = \int\limits_{(w,t) \in Z^d }\int\limits_{(u,v,s) \in Z(w,t)}\Upsilon_j(u,v,w,s,t)s^{\sigma}t^{-\sigma} \lip D^\sigma f,(D^{-\sigma}\phi)_{u,s}\r \\ \times \lip g,\psi_{v,s}^1 \rip \lip h,(D^\sigma\phi)_{w,t} \rip \dfrac{\d s \, \d u \, \d v \, \d t \, \d w}{st}.\end{multline*}
By Proposition \ref{prop:psi-nu}, $\Upsilon_j(u,v,w,s,t)(\frac{s}{t})^{\sigma}$ supplies $\nu'_z \in \Psi^{\lfloor j-\sigma \rfloor,\delta;0,1}_{z} \subset \Psi^{0,\delta;0,1}_{z}$.

The paraproduct term $\Pi_{b_\gamma^0,\gamma}$ follows this same outline since we will not have to apply $D^\sigma$ to a non-cancellative wavelet. 

\subsubsection{Intrinsic Forms $\MSSsig$ and $\pi_b^\sigma$}

Let us begin by constructing the intrinsic form to represent $U_j^3$ (and similarly $U_j^6$). As before,
	\[ \lip D^\sigma h \otimes g, \nu_z \rip \lip f,\phi_z \rip  = \lip h \otimes g, t^{\sigma}D_x^{\sigma} \nu_z\rip \lip D^\sigma f,(D^{-\sigma} \phi)_z \rip. \]
Setting $\nu_z^\sigma=t^{\sigma}D^{\sigma}_y \nu_{z}$ and $\phi'=D^{-\sigma} \phi$, we obtain
	\[ U_j^3(f,g,D^\sigma h) = \int_{Z^d}\lip h \otimes g, \nu_z^\sigma \rip \lip D^\sigma f,\phi'_z \rip \, \d\mu(z) \]
which is dominated by the intrinsic form
	\begin{equation}\label{eq:msssig} \MSSsig(h,g,f) = \int_{Z^d} \sup_{\nu^\sigma_z \in \Psi_z^{\sigma,k,\delta;1,0}} \l h \otimes g,\nu_z^\sigma \r \Psi^{\cals_0}_z( D^\sigma f ) \d\mu(z) \end{equation}
for some new wavelet class $\Psi_z^{\sigma,k,\delta;1,0}$ which we now define. The decay of $\nu_z^\sigma$ can be computed, but it will be asymmetric in the two variables, so let us introduce a new norm
	\begin{equation}\label{eq:psi-sigma} \begin{aligned} \|\psi\|_{\sigma,\eta,\delta} &= \sup_{x\in \R^{2d}} \l x \r^{d+\sigma}(1+|x_2|)^{d+\eta-\sigma}|\psi(x)| \\
		&\quad \quad + \sup_{x\in \R^{2d}} \l x \r^{d+\sigma}(1+|x_2|)^{d+\eta-\sigma}\frac{|\psi(x)-\psi(x+h)|}{|h|^\delta} \end{aligned}\end{equation}
if $\sigma \le d+\eta$. Otherwise $\|\cdot\|_{\sigma,\eta,\delta} = \|\cdot\|_{\star,\eta,\delta}$ from (\ref{eq:star-norm}). Define the associated wavelet class by
	\[ \Psi^{\sigma,k,\delta;1,0}_z = \{ \psi \in C^{k} : \|(\Sy^1_z\Sy^2_z)^{-1}\psi\|_{\sigma,k+\delta,\delta} \le 1 \}. \]
Notice that if $k > \sigma > d$, then $\Psi^{\sigma,k,\delta;i,j}_z = \Psi^{0,\delta;i,j}_z$.
The next lemma will establish that this norm and wavelet class are the correct ones for the modified intrinsic forms.

\begin{lemma}\label{lemma:frac-wave}
Let $0<\sigma < k$. There exists $C>0$ such that
	\[ t^\sigma \|(\Sy^1_z\Sy^2_z)^{-1}D^\sigma_{x_1} \nu_z \|_{\sigma,k+\delta,\delta} \le C.\]
for all $\nu_z \in \Psi_z^{k,\delta;1,0}$.
\end{lemma}

In this way, $U_j^3(f,g,h) \lesssim \MSSsig(h,g,D^\sigma f)$, $U_j^6(f,g,h) \lesssim \MSSsig(h,f,D^\sigma g)$, and similarly for the paraproducts, defining
	\begin{equation}\label{eq:pi-sig} \pi_b^\sigma(f,g,h) = \int_{Z^d} \Psi_z^{\cals_0}(b) \Psi_{z}^{\sigma,k,\delta;1,1}(f,g) \Psi_{z}^{\cals_0}(h) \, \d \mu(z). \end{equation}

\begin{proof}[Proof of Lemma \ref{lemma:frac-wave}] If $\nu_z$ is Schwartz, then this is well-known, see \cite{grafakos2014kato}. Our starting point is from the book \cite[Definition 2.4.5 and Theorem 2.4.6, pp. 127-130]{grafakos2008classical} which states that for all $\phi$ Schwartz and $\sigma>0$,
	\[ \int |\xi|^\sigma \widehat \phi(\xi) \, d\xi = c_{d,\sigma} \int |u|^{-d-\sigma} \phi(u) \, du. \]
The right hand side is well-defined once understood as
	\[ \int_{|u|<1} \frac{\phi(u) - \sum_{|\alpha| \le k} \frac{\phi^{(\alpha)}(0)}{\alpha !} u^\alpha}{|u|^{d+\sigma}} + \int_{|u|>1} \frac{\phi(u)}{|u|^{d+\sigma}} + \sum_{|\alpha| \le k} b(\sigma,k,\alpha) \phi^{(\alpha)}(0) \]
for any $k > \sigma$. Written this way the integrals are all absolutely convergent with bounds depending on $d,\sigma,k$ and linearly on the H\"older constant for $\partial^k\phi$ near $0$ and $\|\phi^{(\alpha)}(0)\|_\infty$. By density, the definition can be extended to functions in $\Psi^{k,\delta;1}_{(0,1)}$. Moreover, replacing $\phi$ by $\phi(\cdot-x)$ for some $x$, we obtain that the first and last terms decay as well as $(1+|x|)^{-(k+d+\delta)}$. The remaining (middle) term is estimated by splitting the integral into the region $|u| > |x|/2$ and $|u| \le |x|/2$.
The first region is controlled by the decay of $\phi$, let us say it is $M>d$, using the fact that the kernel is bounded away from zero since $|u+x| >1$. Then for any $\eta>0$, 
	\[ \int_{|u| > |x|/2|} |\phi(u)| \, du \lesssim (1+|x|)^{d+\eta-M} \int (1+|u|)^{M-d-\eta}|\phi(u)| \, \d u. \]
The other range is limited by $\sigma$ using the fact that in this range $|u+x| \ge |x|-|u| \ge |x|/2$ so that
	\[ \int_{|u|< |x|/2} \frac{|\phi(u)|}{|u+x|^{d+\sigma}} \, \d u \lesssim |x|^{d+\sigma}. \]
To complete the proof we use the translation invariance and homogeneity of the kernel to obtain
	\[ (\Sy^{1}_z)^{-1} (D^\sigma_{x_1} \nu_z)(x_1,x_2)
	=t^{-\sigma} D^\sigma_{x_1} (\Sy^1_z)^{-1}\nu_z(x_1,x_2)\]
for any $\nu_z \in \Psi^{k,\delta;1,1}_z$ and applying the previous estimate to $(\Sy^1_z\Sy^2_z)^{-1}\nu_z(\cdot,x_2)$ to get that the decay is either
	\[ \max\{1,|x_1|,|x_2|\}^{-(\sigma+d)}(1+|x_2|)^{-[(k-\sigma)+d+\delta]}. \]
or \[ \max\{1,|x_1|,|x_2|\}^{-M+d}(1+|x_2|)^{-d}, \] with $M=k+2d+\delta$, whichever is worse.
\end{proof}

\section{Sparse Bounds for Intrinsic Forms}\label{sec:sparse}

\begin{proposition}\label{prop:sparse} Let $\MSS$, $\pi_b$, $\MSSsig$, $\pi_b^\sigma$ be defined  by \eqref{eq:mss}, \eqref{eq:pi}, \eqref{eq:msssig}, and \eqref{eq:pi-sig}. Let $b\in \mathrm{BMO}$.
\begin{itemize}\item[\textrm{1}.]The forms $\MSS$ and $\pi_b$ have sparse $(1,1,1)$ bounds.
\item[\textrm{2}.] The forms $\MSSsig$ and $\pi_b^\sigma$ have sparse $(1,p_2,p_3)$ bounds for any $1 \le p_2,p_3 \le \infty$ with $\frac{1}{p_2}+\frac{1}{p_3} < \frac {\sigma+d} d$.
\end{itemize}
\end{proposition}

Before beginning the proof, we collect some estimates in the following simple lemma. 
\begin{lemma}\label{lemma:max}
Let $\eta \ge \delta>0$, and $Q \subset \R^d$ be a cube with center $c(Q)$ side length $\ell(Q)$. For all $z=(w,t) \in Z^d$ and $\|\Sy^{-1}_z\theta_z\|_{\star,\eta,\delta} \le 1$, if $w \not\in 3Q$ then
	\begin{equation}\label{eq:large-far} \left| \lip f1_Q , \theta_{w,t}\rip \right| \lesssim \l f \r_Q 
			 \frac{\ell(Q)^dt^\eta}{\max\{t,|w-c(Q)|\}^{d+\eta}}. \end{equation}
If in addition $f$ has mean zero, then
	\begin{equation}\label{eq:cancel}  \left|\l f1_Q,\theta_{w,t} \r \right| \lesssim \l f \r_Q \frac{\ell(Q)^{d+\delta}t^{\eta-\delta}}{\max\{t,|w-c(Q)|\}^{d+\eta}}. \end{equation}
On the other hand, if $w \in Q$, then
	\begin{align}\label{eq:out} \left| \lip f 1_{(3Q)^c} ,\theta_{z} \rip \right| &\lesssim  
			\left(\frac{t}{\ell(Q)} \right)^{\eta} \inf_{u \in Q} M(f)(u).  
	\end{align}
\end{lemma}
\begin{proof}
We begin with (\ref{eq:large-far}). 
If $x \in Q$ and $w \not\in 3Q$, then $|x-w| \ge \frac 12 |w-c(Q)|$. 
Since $\theta_{w,t}(x) \le t^{-d}(1+\frac{|x-w|}{t})^{-(d+\eta)}$, this implies
	\[ \int_{Q} |f(x) \theta_{w,t}(x)| \, \d x \lesssim \l f \r_{Q} |Q| \frac{t^\eta}{\max\{t,|w-c(Q)|\}^{d+\eta}}. \] 
If $f$ has mean zero, then
	\[ |\l f1_Q,\theta_{w,t} \r| = \left|\int_Q f(x)[\theta_{w,t}(x)-\theta_{w,t}(c(Q))] \, \d x \right| \]
	\[\lesssim \ell(Q)^\delta \int_Q |f(x)|t^{-d-\delta}\left(1+\frac{|c(Q)-w|}{t}\right)^{-(d+\eta)} \, \d x \le \l f \r_{Q} \frac{\ell(Q)^{d+\delta}t^{\eta-\delta} }{\max\{t,|w-c(Q)|\}^{d+\eta}}. \]

For (\ref{eq:out}), decompose the integral into dyadic annuli $A_j = B(u,2^{j+1}\ell(Q)) \backslash B(u,2^j\ell(Q))$ for any $u \in Q$. We can skip the ball $B(u,\ell(Q))$ since $B(u,\ell(Q)) \subset 3Q$.
Therefore,
	\[\begin{aligned} |\lip f, \theta_{w,t} \rip| &\le \sum_{j=0}^{\infty} \int_{A_j} |f(x) \theta_{w,t}(x)|  \, \d x
			\lesssim \sum_{j=0}^\infty |A_j| \l f \r_{1,A_j} t^\eta (2^j\ell(Q))^{-(d+\eta)}
			\lesssim \left( \frac{t}{\ell(Q)}\right)^\eta Mf(u). \end{aligned}\]
This completes the proof of the lemma.\end{proof}

There are only two points at which we will need to distinguish among the four forms, so we will use $\A$ to represent any one of them except at these two crucial points. The key property of $\A$---the shared property of $\MSS$, $\pi_b$, $\MSSsig$, and $\pi_b^\sigma$---is the bound
	\begin{equation}\label{eq:a-form} \left| \A(f_1,f_2,f_3) \right| \lesssim \int_{Z^d} \Psi^{0,\delta;1}_z(f_1)\Psi^{0,\delta;1}_z(f_2)\Psi^{0,\delta;1}_z(f_3) \, \d \mu(z), \quad \Psi_z^{0,\delta;1}(f) := \sup_{\theta_z \in \Psi_z^{0,\delta;1}} | \l f ,\theta_z \r|. \end{equation}
All four forms also satisfy the bound $\A(f_1,f_2,f_3) \lesssim \prod_{i=1}^3 \|f_i\|_{q_i}$ where $1 < q_i < \infty$ and $\sum q_i^{-1} = 1$, though for different reasons. This is the first point at which we consider each form separately, since the bound in (\ref{eq:a-form}) above is not $L^p$ bounded. For $\Pi$ and $\Pi^\sigma$, they can be bounded by $\|Mf_1\|_{L^{q_1}} \|Sf_2\|_{L^{q_2}} \|Sf_3\|_{L^{q_3}}$ where $M$ and $S$ are modified maximal and intrinsic square functions for which the \textit{linear} $L^p$-mapping properties are well-known \cite{diplinio20}. For the paraproducts, it is a little more complicated, yet still within the realm of the standard linear theory, so we do not prove it here.

\begin{proof}[Proof of Proposition \ref{prop:sparse}] {
 In the course of this proof, the following additional notations will find use. First, to unify, the tuple  $\vec r=(r_1,r_2,r_3)$ is set as follows:
\[\vec  r\coloneqq 
\begin{cases}
(1,1,1) & \A \in\{\MSSsig, \pi_b^\sigma\}, \\
(1,p_2,p_3) & \A \in\{\MSS  \pi_b\}.
\end{cases} \]There is no loss in generality with the assumption  $1\leq p_2,p_3 \le 2$ since $\frac{\sigma+d}{d} >1$ and $\vec p$ sparse bounds imply $\vec p'$ sparse bounds for any $\vec p \le \vec p'$.
  Let  $\mathcal D$ be the standard dyadic system on $\mathbb R^d$. For a cube $Q \in \mathcal D$, $T(Q)$ is the Carleson box $T(Q) = (0,\ell(Q)] \times Q \subset (0,\infty) \times \R^d$.  If $\mathcal E\subset \mathcal D$ is a pairwise disjoint cover of $E\subset \mathbb R^d$.  take $T(E)\
\coloneqq \cup\{T(Q): Q \in \mathcal E\} $. Finally, if $F\subset (0,\infty) \times \R^d $. the truncated operators $\A_F$ are defined for any $F \subset Z^d$ by integrating over only $F$ in (\ref{eq:a-form}).

 Our task is to prove \eqref{e:sparsegen} holds for all triples $f_j\in L^\infty(\R^d)$ with compact support. Fix such a triple and take $Q_0\in \mathcal D$ with the property that $\supp\, f_j\subset 3Q_0$ for all $j=1,2,3$. The proof is iterative in nature, and we begin   the main step of the iteration by defining an exceptional subset. Let
\[
E\coloneqq \bigcup_{j=1}^3 \left\{x\in \R^d : \mathrm{M}_{r_j} f_j (x) > C \l f_j \r_{r_j,3Q_0} \right\}.
\]
  For $C$ large, by the maximal inequality, $|E| \le 2^{-4d} |Q_0|$. 
Let now $\mathcal E$ be the maximal elements of the collection $\{Q\in \mathcal D: 9Q \subset E\}$. Clearly $\mathcal E$ is a pairwise disjoint cover of $E$. Moreover, the stopping nature of $Q\in \mathcal E$ yields the property
\[\inf_{Q} \mathrm{M}_{r_j} f_j \lesssim  \l f_j \r_{r_j,3Q_0} 
\]
uniformly over $Q\in \mathcal E$ and $j=1,2,3$. This property will be  tacitly used  throughout the proof.
We use $E$ to induce the decomposition $\A = \A_{T(E)^c} + \A_{T(E)}$, whose terms we estimate separately.} Let us begin with $\A_{T(E)}$. Break this up as
	\[ \A_{T(E)}(f_1,f_2,f_3) = \sum_{Q \in \mathcal E} \sum_{\vec g} \A_{T(Q)}(g_1,g_2,g_3) \]
where each $\vec g = (g_1,g_2,g_3)$ runs over $2^3$ possibilities where each $g_j$ is either $f_j1_{3Q}$ (in) or $f_j1_{{(3Q)^c}}$ (out).
We leave alone the term consisting entirely of ``in'' functions. ``Out'' functions are good so let us assume $g_1$ is out and the others are in. This is the second point at which we distinguish among the four forms for $\A$. 
For $\MSS$ and $\pi_b$,  $\vec r=(1,1,1).$ We obtain, applying (\ref{eq:out}) to $g_1$ with $\eta=d+\delta$,
	\[ \int_{T(Q)} \prod_{j=1}^3 |\l g_j,\theta_{w,t} \r| \dfrac{\d w \, \d t}{t} \lesssim \int_0^{\ell(Q)} \frac{t^{d+\delta}}{\ell(Q)^{d+\delta}} \inf_{u \in Q} \M {{g_1}(u)} \l |f_2|, \int_Q |\theta_{w,t}| \d w \r   \|f_3\|_{L^1}t^{-d} \dfrac{\d t}{t} \]
	\[\lesssim |Q|\inf_{Q} \M(1_{{{(3Q)^c}}}f_1)\l f_2\r_{3Q} \l f_3\r_{3Q}\]
where we used the fact that $\int_Q |\theta_{w,t}(x)| \, \d w \le \|\theta_{x,t}\|_{L^1} \lesssim 1$.
Therefore, summing over $Q \in \mathcal E$,
	\[ \sum_{Q \in \mathcal E}   |Q|\inf_Q \M(1_{{{(3Q)^c}}}f_1)\l f_2\r_{1,3Q} \l f_3\r_{1,3Q} \lesssim \prod_{j=1}^3  |Q_0| \l f_j \r_{r_j,3Q_0}. \]
One can verify the same result for two out functions, and when all three are out, use the improved decay to obtain
	\[ \int_{T(Q)} \frac{t^{2(d+\delta)}}{\ell(Q)^{2(d+\delta)}} \prod_{j=1}^3  \inf_Q \M g_j \dfrac{\d w \, \d t}{t} \lesssim |Q| \prod_{j=1}^3  \l f_j \r_{r_j, 3Q_0}. \]
For $\MSSsig$ and $\pi_b^\sigma$, consider the case where there is only one out function and it is in the first position; this is the worst case  as the other cases actually result in the same situation above with the power $d+\delta$.  The second wavelet in $\MSSsig$ has decay greater than $d+\delta$: cf. (\ref{eq:psi-sigma}) since $k-\sigma>0$. Thus for any $q\ge 1$, setting $\chi_t(x) = t^{-d}(1+\frac{|x|}t)^{d+\delta}$, $\|\theta_{w,t}\|_{L^q} \le \|\chi_t\|_{L^q} \lesssim t^{d(1/q-1)}$. 
{ For $j=2,3$ set
\[
z_j={\left(\textstyle \frac{1}{p_2}+\frac{1}{p_3}\right)p_j \ge p_j}, \qquad q_j\coloneqq \left(\textstyle \frac{1}{z_j}+1-\frac{1}{p_j} \right)^{-1} \geq 1\] and apply Young's inequality to obtain
	\[ \left(\int_Q |\l g_j, \theta_{w,t} \r|^{z_j} \, \d w\right)^{1/r_j} \le \||g_j| * \chi_t\|_{L^{r_j}} \le \|g_j\|_{L^{p_j}}\|\chi_t\|_{L^{q_j}} \lesssim \|g_j\|_{L^{p_j}} t^{d(1/q_j-1)}\] Applying H\"older's inequality with exponents $z_2$ and $z_3$
along with the above estimate gives
	\[ \int_Q |\lip g_2,\theta_{w,t} \r ||\l g_3,\theta_{w,t} \r |\, \d w 
	\lesssim t^{d(1-1/p_2-1/p_3)}\prod_{j=2}^3 \|g_j\|_{L^{p_j}(Q)}.\]
}
Now we use the fact that $\frac{1}{p_2}+\frac{1}{p_3} < \frac{\sigma+d}{d}$. This implies $\sigma + d(1-\frac{1}{p_2}-\frac{1}{p_3}) > 0$ so that
	\[ \int_{T(Q)} \prod_{j=1}^3 |\l g_i,\theta_{w,t} \r | \dfrac{\d w \, \d t}{t} \lesssim \int_0^{\ell(Q)} \frac{t^{\sigma}}{\ell(Q)^{\sigma}} \inf_{u \in Q} \M(1_{{{(3Q)^c}}}g_1)(u) t^{d(1-1/p_2-1/p_3)}\|g_2\|_{L^{p_2}} \|g_3\|_{L^{p_3}} \dfrac{\d t}{t} \]	
	\[ \lesssim \ell(Q)^{d(1-1/p_2-1/p_3)} \left(\inf_{Q} \M(1_{{{(3Q)^c}}}g_1 \right) \|g_2\|_{L^{p_2}} \|g_3\|_{L^{p_3}} = |Q| \left(\inf_{Q} \M g_1 \right)\l g_2 \r_{p_2,Q} \l g_3 \r_{p_3,Q}. \]
Up until now, we have shown
	\[ \begin{aligned} \A(f_11_{3Q_0},f_21_{3Q_0},f_31_{3Q_0}) &\lesssim \sum_{Q \in \mathcal E} \A_{T(Q)}(f_11_{3Q},f_21_{3Q},f_31_{3Q}) + |Q_0| \prod_{j=1}^3 \l f_j \r_{r_j,3Q_0} \\&\quad \quad \quad + \A_{T(E)^c}(f_11_{3Q_0},f_21_{3Q_0},f_31_{3Q_0})  \end{aligned} \]
We create the sparse collection by applying the same argument to each $Q \in \mathcal E$ as if it were $Q_0$; see \cite{conde-alonso17} for details. Iterating this, we will be done once we show 
	\begin{equation}\label{eq:remainder} \A_{T(E)^c}(f_11_{3Q_0},f_21_{3Q_0},f_31_{3Q_0}) \lesssim |Q_0|\prod_{j=1}^3 \l f_j \r_{r_j,3Q_0}. \end{equation}
Perform a Calder\'on-Zygmund decomposition of each $f_j = g_j + b_j = g_j + \sum_{Q \in \mathcal E} b_j^Q$ with respect to the collection of cubes $\mathcal E$ and at the level $ \l f_j \r_{r_j,3Q_0}$. The good functions $g_j$ are estimated using the ${{L^3 \times L^3  \times L^3  }}$ boundedness,
	\[ \A_{T(E)^c} (g_1,g_2,g_3) \le \A(g_1,g_2,g_3) \lesssim |Q_0| \prod_{j=1}^3 \l f_j \r_{r_j,3Q_0}. \]
The remaining terms all have at least one bad term. Let us say it is in the first argument. The functions in the other two arguments will be estimated using
	\begin{equation}\label{eq:gb}| \lip g_j, \theta_{w,t} \rip|, \, | \lip b_j,\theta_{w,t} \rip| \lesssim  \l f_j \r_{r_j,3Q_0} \quad (w,t) \not\in T(E). \end{equation}
For the good functions, (\ref{eq:gb}) is an obvious consequence of $\|g_j\|_\infty \le \l f_j \r_{r_j,3Q_0}$. The bad one requires some work. Decompose the sum into two regions: 
	\[ I = \{ Q \in \mathcal E : w \not\in 3Q \}, \quad II = \{Q \in \mathcal E : 9t > \ell(Q), \, w \in 3Q\}.\]
Since $(w,t) \not\in T(E)$, we claim $\mathcal E = I \cup II$. Indeed, for each $Q \in \mathcal E$, if $w \in 3Q$, then $w \in Q'$ for some $Q' \in \mathcal E$ with $\ell(Q') \ge \frac 19 \ell(Q)$. Therefore $9t \ge \ell(Q)$. Considering $II$ first,
	\[ \sum_{Q \in II} |\lip b_j^Q,\theta_{w,t} \rip| \le \sum_{Q \in II} \frac{|Q|}{t^d} \l b_i^Q \r_{Q} \le \l f_j \r_{r_j,3Q_0} t^{-d} \sum_{Q \in II} |Q|.\]
But the cubes are disjoint and contained in the cube centered at $w$ with side length $18t$. This means $\sum_{Q \in II}|Q| \lesssim t^d$. For $I$, the estimate immediately follows from Lemma \ref{lemma:max} if we can establish
	\begin{equation}\label{eq:cubes} \sum_{Q \in I } \frac{|Q| \min\{\ell(Q),t\}^{\delta} }{|w-c(Q)|^{d+\delta}} \lesssim 1. \end{equation}
Let us now complete the proof, postponing (\ref{eq:cubes}) until the end. For the same reason that $\mathcal E = I \cup II$ above, 
	\[ T(E)^c \subset \{ w\not\in 3Q\} \cup \{9t \ge \ell(Q)\}=:T^*(Q)^c\]
for any $Q \in \mathcal E$. Let $h_j$ be either $g_j$ or $b_j$ so that by (\ref{eq:gb}) $|\l h_j,\theta_z \r| \lesssim \l f_j \r_{3Q_0}$ for $z \in T(E)^c$. Using the first two statements from Lemma \ref{lemma:max},
	\[ \begin{aligned} &\quad \A_{T(E)^c}(b_1,h_2,h_3) \lesssim \sum_{Q \in \mathcal E} \int_{T(E)^c} |\l b_1^Q,\theta_{z} \r \l h_2,\theta_{z} \r \l h_3,\theta_z \r| \d \mu(z) \\
		&\lesssim \l f_2 \r_{r_2,3Q_0} \l f_3 \r_{r_3,3Q_0} \sum_{Q \in \mathcal E} \l b_i^Q \r_{Q}|Q| \left(\int_{\frac 19 \ell(Q)}^\infty \int_{3Q} t^{-d} \frac{\d w\, \d t}{t} + \int_{T^*(Q)^c} \frac{\min\{t,\ell(Q)\}^\delta}{\max\{t,|w-c(Q)|\}^{d+\delta}} \dfrac{ \d w \, \d t}{t} \right)\\
		&\lesssim  \l f_1 \r_{r_1,3Q_0} \l f_2 \r_{r_2,3Q_0} \l f_3 \r_{r_3,3Q_0} \sum_{Q \in \mathcal E} |Q|  \lesssim |Q_0| \prod_{i=1}^3 \l f_i \r_{3Q_0}.
	\end{aligned} \]
In the third inequality, we used the fact that for any $\delta>0$, and $Q$ cube,
	\[ \begin{aligned} &\quad\int_{T^*(Q)^c} \frac{\min\{\ell(Q),t\}^\delta}{\max\{t,|w-c(Q)|\}^{d+\delta}} \dfrac{\d w \, \d t}{t}  \lesssim \int_{\frac 19 \ell(Q)}^{\infty} \int\limits_{|w-c(Q)| \le t} \frac{\ell(Q)^\delta}{t^{d+\delta}} \frac{\d w \, \d t}{t}\\
	& + \int_{\frac 19 \ell(Q)}^{\infty} \int\limits_{\{|w-c(Q)| > t\}}\frac{\ell(Q)^\delta}{|w-c(Q)|^{d+\delta}} + \int_{0}^{\frac 19 \ell(Q)} \int_{(3Q)^c} \frac{t^\delta}{|w-c(Q)|^{d+\delta}} \dfrac{\d w \, \d t}{t} \\
	& \lesssim \int_{\ell(Q)}^{\infty} \frac{\ell(Q)^\delta}{t^{\delta}} + \int_{0}^{\ell(Q)} \frac{t^\delta}{\ell(Q)^{\delta}} \dfrac{\d t}{t}  \lesssim 1. \end{aligned} \]
This is the continuous version of (\ref{eq:cubes}) so it is established.
\end{proof}

\section{General Cases}\label{sec:asym}
The results and arguments can be almost immediately extended to $m$-linear operators and the associated $(m+1)$-linear forms. At the same time, we would like to generalize to kernels which have varying degrees of smoothness in each variable.
This second generalization is motivated by the fact that the assumptions in the first representation theorem were symmetric in $\Lambda$ and both its adjoints. However, in the Sobolev mapping theorem, we saw that the conditions were asymmetric, and in fact some of the estimates on the adjoint terms were a bit too good. So, we give a representation theorem which is asymmetric and allows us to prove the Sobolev result under weaker assumptions. We must slightly alter the definitions above.

\subsection{Singular Integrals} Let $\vec 1_{d}=(1,\ldots, 1)\in \R^d$. 
Given $\vl = (\ell_0,\ell_1,\ldots,\ell_{m}) \in \N^{m+1}$, a function $K\in L^1_{\mathrm{loc}}( \R^{(m+1)d}\setminus \R\cic{1}_{(m+1)d})$ is an $(\vl,\delta)$ SI (singular integral) kernel if there exist $C,\delta>0$ such that for all $0 \le |\kappa| \le \ell_i$, 
	\[ |\nabla_{x_i}^\kappa K(x_0,x_1,\ldots,x_m)| \le \dfrac{C}{(\sum_{j \ne i}|x_i-x_j|)^{md+|\kappa|}} \]
	\[ |\nabla_{x_i}^\kappa \Delta^i_h K(x_0,x_1,\ldots,x_m)| \le \dfrac{C|h|^\delta}{(\sum_{j \ne i}|x_i-x_j|)^{md+|\kappa|+\delta}}.\]
We say $\Lambda$ is an $(\vl,\delta)$ $(m+1)$-linear SI form if
	\[ \int_{(\R^d)^{m+1}} K(x_0,x_1,\ldots,x_m) \prod_{j=0}^m f_j(x_j) \, \d x = \Lambda(\vec f) \]
for all $\vec f = (f_0,f_1,\ldots,f_m) \in\cals^{m+1}$ with $\cap_{i=0}^m \supp f_i = \varnothing$ and an $(\vl,\delta)$ SI kernel $K$. Notice that an $(\vl,\delta)$ SI form is an $(\vl',\delta')$ form for any $\vl' \le \vl$ and $\delta' \le \delta$.

\subsection{Calder\'on-Zygmund Forms}
It is useful at this point to define the adjoints of an $(m+1)$-linear form. For each $i=0,1,\ldots,m$ 
	\[ \Lambda^{i*}(\vec f) = \Lambda(f_i,f_1,f_2,\ldots,f_{i-1},f_0,f_{i+1} \ldots, f_m). \]
In other words, $\Lambda^{i*}$ permutes $f_0$ and $f_i$ and it is clear that $\Lambda^{0*}=\Lambda$.

\subsubsection{Paraproducts}
We say $\Lambda$ has $\vec 0$-th order paraproducts if for each $i=0,1,\ldots,m$, there exists $b_0^i$ in BMO such that
	\[ \Lambda^{i*}(\psi,1,1,\ldots,1) = \lip b_0^i,\psi \r \]
for all $\psi \in \cals_0$. For $\vec j = (j_0,j_1,\ldots,j_m) \in (\N^{m})^{m+1}$, we define the $\vec j$-th order paraproducts inductively. We again use the paraproduct forms, now defined for any $\gamma =(\gamma_1,\ldots,\gamma_m) \in (\mathbb N^{d})^m$ and $b \in \BMO$ by
	\[ \Pi_{b,\gamma}(\vec f) = \int_{Z^d} \l b, (\partial^{-(\gamma_1 + \cdots + \gamma_m)}\phi)_{z}\r \prod_{i=1}^m \l f_i, \varu_z^{\gamma_i} \r \l f_0, \phi_z \r \, \d \mu(z). \]
Suppose for each $i=0,1,\ldots,m$, $\Lambda$ has paraproducts $b_\gamma^{i}$ for all $(|\gamma_1|,\ldots,|\gamma_m|) < j_i$. Then, we say $\Lambda$ has $\vec j$-th order paraproducts if for each $(|\gamma_1|,\ldots,|\gamma_m|)=j_i$, there exist $b_{\gamma}^i \in \BMO$ such that for all $\psi \in \cals_{|j_i|}$,
	\[ \Lambda_{\vec j} := \Lambda - \sum_{i=0}^m \sum_{\substack{(|\kappa_1|,\ldots,|\kappa_m|) \\ < j_i}} \Pi^{i*}_{{b_\kappa^i},\kappa} \]
satisfies
	\[ \Lambda_{\vec j}^{i*}(\psi,x_1^{\gamma_1},x_2^{\gamma_2},\ldots,x_m^{\gamma_m}) = \lip b_\gamma^i,\partial^{-(\gamma_1 + \cdots + \gamma_m)}\psi \rip.\]
Under this definition, one can verify by induction that $\Lambda_{\vec j}$ has vanishing paraproducts of all orders $<\vec j$.

\begin{definition}
Let $\vl \in \N^{m+1}$ and $\vec k =(k_0,k_1,\ldots,k_m) \in (\N^{m})^{m+1}$ with
	\[ |k_i| \le \ell_i, \quad i=0,1,\ldots,m. \]
A $(\vl,\delta)$ SI form $\Lambda$ is called a $(\vec k,\delta)$ CZ (Calder\'on-Zygmund) form if it has paraproducts up to order $\vec k$ and satisfies the Weak Boundedness Property: There exists $C>0$ such that
	\[ t^{md} \Lambda(\psi^0_z,\psi_z^1,\ldots,\psi_z^m) \le C \]
for all $\psi^i_z \in \Psi_z^{0,\delta;1}$ supported in the ball $B(w,t)$.
\end{definition}

\subsubsection{Wavelet Forms}
The trilinear wavelet forms and wavelet classes must also be extended to the $m$-linear setting. Extending the norm $\|\cdot\|_{\star,\eta,\delta}$ to functions defined on $(\R^d)^m$, the wavelet classes $\Psi^{k,\delta;\iota}_z$ are the collection of all $\varphi \in C^k(\R^{md})$ such that
	\[ t^{|\gamma|} \| (\Sy_z^1\cdots\Sy_z^m)^{-1}\partial^\gamma \varphi\|_{\star,k+\delta,\delta} \lesssim 1\quad \mbox{for } \gamma \in \N^{dm}, \quad |\gamma| \le k \]
and $\iota \in \{0,1\}^m$ controls the cancellation in the obvious way. The main case we will need is $\iota = (1,1,\ldots,1,0)$ in which case $\varphi$ satisfy
	\[ \int_{\R^d} x_m^{\gamma_m} \varphi(x_1,\ldots,x_m) \, \d x_m = 0 \quad \mbox{for } |\gamma| \le k. \]
A $(j,\delta)$ wavelet form $U_j$ is now defined, for some $\nu_z \in \Psi^{j,\delta;(1,\ldots,1,0)}_z$, by
	\[ U_j(\vec f) = \int_{Z^d} \left \l \otimes_{i=1}^m f_i, \nu_z \right \r \l f_0, \phi_z \r \, \d \mu(z). \]

\begin{theorem}\label{thm:vec}
Let $0<\eta<\delta$ and $\vec k =(k_0,k_1,\ldots,k_m) \in (\N^m)^{m+1}$. Let $k_i^*$ be the smallest entry of $k_i$ and let $\Lambda$ be a $(\vec k,\delta)$ CZ form. Then there exists $(j,\eta)$-smooth wavelet forms $U^{i,\pi}_j$, and paraproduct forms $\Pi_{\gamma}$ such that
	\[ \Lambda(\vec f) = \sum_{i=0}^m \left[ \sum_{j=k_i^*}^{|k_i|} \sum_{\pi \in S^{m}} U^{i,\pi}_j(\pi(f_0,\ldots,f_{i-1},f_{i+1},\ldots,f_m),f_i) + \sum_{(|\gamma_1|,\ldots,|\gamma_m|) \le k_i}\Pi^{i*}_{b_\gamma^i,\gamma}(\vec f) \right] \]
for all $\vec f = (f_0,f_1,\ldots,f_m) \in\cals^{m+1}$.
\end{theorem}

If $m=2$, the proof is the same as before, only the steps with $II$ and $III$ are carried out as if $(k_1,k_2)$ is replaced by $k_1$ and $k_2 \in \N^2$. For larger $m$, we outline the necessary modifications. Again, decompose $\Lambda(\vec f)$ using the Calder\'on formula (\ref{eq:calderon}) $m+1$ times to get
	\[ \Lambda(\vec f) =  \int_{(Z^d)^{m+1}}  \Lambda(\phi_{z_0},\ldots,\phi_{z_m}) \prod_{i=0}^m \l f_i, \phi_{z_i} \r  \, d\mu(z_i). \]
Split $(Z^d)^{m+1}$ into $m+1$ regions $Z_i = \{ z_i = (w_i,t_i) : t_i = \min_\ell t_\ell \}$ and each $Z_i$ again into $Y_{i,j}=\{ t_j = \min_{\ell \ne i} t_\ell \}$. On each $Y_{i,j}$ use Lemma \ref{lemma:calderon} $m-1$ times to bring the integration down to the two scales $t_j > t_i$. In this way, 
	\[ \Lambda(\vec f) = \sum_{i=0}^m \sum_{j \ne i} \sum_{\vec \psi, \, \tilde \psi} \,  \int\limits_{t_i >0} \int\limits_{t_j > t_i} \int\limits_{(\R^d)^{m+1}}  \Lambda^{i*}(\phi_{w_i,t_i}, \vec \psi) \l f_i, \phi_{w_i,t_i} \r \l f_j,\phi_{w_j,t_j} \r \prod_{\ell=0, \ell \ne i,j}^m \l f_\ell,\tilde \psi_{w_\ell,t_j} \r \dfrac{\d w \, \d t_j \, \d t_i}{t_jt_i}. \]
Each $\vec \psi$ is a vector of $m$ functions where one entry is the mother wavelet $\phi_{w_j,t_j}$ and the other ones are either $\psi^1_{w_\ell,t_j}$ or $\psi^3_{w_\ell,t_j}$ --- the cancellative functions from Lemma \ref{lemma:calderon}. $\tilde \psi$ is either $\psi^2$ or $\psi^4$.
This gives $(m+1) \times m \times 2^{m-1}$ terms which correspond to the 12 terms $I + II + III$ from the proof of Theorem \ref{thm:sym}. Each summand is handled in the same way as $\sigma_1$ in (\ref{eq:sigma0-4}) above.
The kernel estimates and wavelet averaging lemma (Lemmas \ref{lemma:star} and \ref{lemma:U}) can be easily reproduced in the same way as in the bilinear case.

Our extension to the nonsymmetric case generalizes  results of \cite{frazier88,benyi2003bilinear} to forms whose paraproducts of lower orders do not vanish. In particular, we obtain Sobolev bounds when the kernel only has extra smoothness in one of the $m+1$ variables.

\begin{cor}
Let $k_0 \in \N^m$, $\vec k = (k_0,0,\ldots,0)$, and $\Lambda$ be a $(\vec k,\delta)$ CZ form with 
	\[ D^{k_0^*-|\gamma|}b_\gamma^0 \in \BMO, \quad \mbox{for } |\gamma| \le k_0^*\]
where $k_0^*$ is the minimum entry of $k_0$. Then, for $\vec p \in P_\circ$, $p=p(\vec p)$, and $\vec v \in A_{\vec p}$,
	\[ \begin{aligned} \|T(f_1,\ldots,f_m)\|_{\dot W^{k_0^*,p}(\frac{1}{v_{m+1}})} &\lesssim [\vec v]_{A_{\vec p}}^{\max\{p_i',p\}} \sum_{|\vec j| \le k_0^*} \prod_{i=1}^m \|f_i\|_{\dot W^{j_i,p_i}(v_i)}, \\
	\|T(f_1,\ldots,f_m)\|_{W^{k_0^*,p}(\frac{1}{v_{m+1}})} &\lesssim [\vec v]_{A_{\vec p}}^{\max\{p_i',p\}}\prod_{i=1}^m \|f_i\|_{W^{k_0^*,p_i}(v_i)}. \end{aligned} \]
If in addition, $b_\gamma^0=0$ for $(|\gamma_1|,|\gamma_2|,\ldots,|\gamma_m|)<k_0$, then
	\[ \|T(f_1,\ldots,f_m)\|_{\dot W^{k_0^*,p}(\frac 1{v_{m+1}})} \lesssim [\vec \_{A_{\vec p}}^{\max\{p_i',p\}} \sum_{j=1}^m \|f_j\|_{\dot W^{k_0^*,p_j}(v_j)} \prod_{i\ne j} \|f_iv_i\|_{\dot L^{p_i}(\R^d)}. \] 
\end{cor}

\section{Comments and Further Questions}\label{sec:comments}
We first discuss the laborious definition of the paraproducts introduced here (Section \ref{subsec:para}) and in \cite{diplinio20}. The reader might object to this definition because, by looking at any SI form, one cannot immediately tell whether it has paraproducts of, let us say, order (1,0), even after constructing $b_0$ and subtracting $\Pi_{b_0}$.

It may be proposed that one may more immediately test $\Lambda(x,1,\psi)$ than $\Lambda_{1,0}(x,1,\psi)$. However, we do not know whether $\Lambda(x,1,\psi)$ has anything to do with the boundedness properties of $\Lambda$. A first example is the form $\Pi_{b,0}$. As shown above, it is enough for $b,Db \in \BMO$ for $\Pi_{b,0}:W^{1,4} \times W^{1,4} \to W^{1,2}$. However, using the ideas of Calder\'on-Toeplitz operators \cite{rochberg1990toeplitz,nowak1993calderon}, it can be shown that $\Pi_{b,0}(x,1,\psi) \sim \lip xb, \psi \rip + \cdots$ and we see no reason why $D(xb) \in \BMO$ should imply $b,Db \in \BMO$ or vice versa. If some real connection could be realized between $\Pi_{b,0}(x,1,\phi)$ and $b,|\nabla b|$, then we could simplify the definition of paraproducts. We also refer to the paper \cite{wang97} where this iterative definition is avoided, however one must pay a price in the testing condition, so that $T(x^\gamma)$ is replaced by $T((x-w)^\gamma)$ for infinitely many $w$.

Secondly, we would like to remark that our results may be   extended to the full spectrum of smoothness spaces, say Triebel-Lizorkin and Besov scales, by simply  adjusting  the procedures of Subsection \ref{subsec:frac-sob} to handle the corresponding smoothness norm. In fact, our framework is particularly apt to handle spaces characterized by  wavelet coefficient estimates such as those of  Besov or Triebel-Lizorkin  type.
One can obtain some negative Sobolev space results of the type $T:W^{-k,p_1}(v_1) \times W^{k,p_2}(v_2) \to W^{-k,p}(v)$ by applying our theorems to $T^{*1}$. Using $T^{*2}$ would exchange the two input spaces.  However, we do not know how to obtain $T: \prod_{i=1}^m W^{-k,p_i} \to W^{-k,p}$ with our methods, except when $m=1$.

Finally,  the constraint $\frac{1}{r_1} + \frac{1}{r_2} < \frac{\sigma+d}{d}$ in the    sparse domination result of Proposition \ref{prop:sparse}, which was the main ingredient leading to   the fractional Sobolev space bound of Theorem \ref{thm:sob-frac}, is sharp up to the equality possibly holding.
Indeed, taking $f(x)=\e^{10ix}\phi(x)$ for $\widehat \phi \in C^{\infty}_0( B(0,1))$, $D^\sigma (f \bar f) \sim (1+|x|)^{d+\sigma}$ for large $x$. For $g= 1_{B(0,2^{k+1})} - 1_{B(0,2^{k})}$,
	\[ \l D^\sigma(f\bar f),g \r \sim 2^{kd} 2^{-k(d+\sigma)}. \]
However, if one had a sparse bound of the form $\l D^\sigma(f \bar f),g \r \lesssim \sum_{Q} |Q| \l D^\sigma f \r_{r_1,Q} \l f \r_{r_2,Q} \l g \r_{r_3,Q}$, then $\l D^{\sigma}(f\bar f),g \r$ would be controlled by
	\[ 2^{kd} \l D^\sigma f \r_{r_1,B(0,2^{k+5})} \l f \r_{r_2,B(0,2^{k+5})} \l g \r_{r_3,B(0,2^{k+1})} \lesssim 2^{kd} 2^{-kd/r_1} 2^{-kd/r_2} \]
so that $d(1/r_1+1/r_2)\le d+\sigma$, i.e. $\frac{1}{r_1} + \frac{1}{r_2} \le \frac{d+\sigma}{d}$.


\bibliography{journal-abbrv-short,refs-wrt}
 \bibliographystyle{amsplain}

\end{document}